\documentclass[12pt]{article}
\usepackage{amsmath,enumerate,amsfonts,amssymb,color,graphicx,amsthm}
\usepackage{hyperref}
\usepackage{todonotes}
\usepackage[normalem]{ulem}
\usepackage{dsfont}

\def\RR{{\mathbb R}}

\newtheorem{theorem}{Theorem}[section]
\newtheorem{proposition}[theorem]{Proposition}

\newtheorem{lemma}[theorem]{Lemma}

\DeclareFontFamily{OT1}{rsfs}{}
\DeclareFontShape{OT1}{rsfs}{m}{n}{ <-7> rsfs5 <7-10> rsfs7 <10-> rsfs10}{}
\DeclareMathAlphabet{\mycal}{OT1}{rsfs}{m}{n}

\newcounter{marnote}

\def\bhv{{\mathbf{\hat v}}}
\newtheorem{thm}{Theorem}[section]

\newtheorem{lem}[thm]{Lemma}
\newtheorem{prop}[thm]{Proposition}
\theoremstyle{definition}
\newtheorem{rem}[thm]{Remark}

\numberwithin{equation}{section}

\newcommand{\bremark}{\begin{rem} \textup}
\newcommand{\eremark}{\end{rem} }

\newcommand{\cuad}{{\sqcap\kern-.68em\sqcup}}

\newcommand{\la}{\lambda}

\newcommand{\R}{{\mathbb{R}}}

\renewcommand{\rho}{\varrho}
\renewcommand{\theta}{\vartheta}
\newcommand{\pa}{\partial}
\newcommand{\eps}{\varepsilon}
\newcommand{\s}{\sigma}
\renewcommand{\S}{\mathbb{S}}

\def\RHSA{{P}}
\def\RHSB{{Q}}

\def\vone{{u}}
\def\vtwo{{v}}

\begin{document}

%----------------------------------------------------------------------------%
\title{Moving planes for domain walls in a coupled system}
\author{Amandine Aftalion\thanks{Ecole des Hautes Etudes en Sciences Sociales, CNRS UMR 8557, Centre d'Analyse et de Math\'ematique Sociales, 54 Boulevard Raspail, 75006 Paris, France. Email: amandine.aftalion@ehess.fr}~, Alberto Farina\thanks{LAMFA, CNRS UMR 7352, Universit\'ee de Picardie Jules Verne, 33 rue Saint-Leu, 80039 Amiens, France. E-mail: alberto.farina@u-picardie.fr}~, Luc Nguyen\thanks{Mathematical Institute and St Edmund Hall, University of Oxford, Andrew Wiles Building, Radcliffe Observatory Quarter, Woodstock Road, Oxford OX2 6GG, United Kingdom. Email: luc.nguyen@maths.ox.ac.uk}}

\maketitle

\begin{abstract}
The system leading to phase segregation in two-component Bose-Einstein condensates can be generalized to hyperfine spin states with a Rabi term coupling. This leads to domain wall solutions having a monotone structure for a non-cooperative system.
We use the moving plane method  to prove monotonicity and one-dimensionality of the phase transition solutions. This relies on totally new estimates  for a type of system for which no Maximum Principle a priori holds. We also derive that one dimensional solutions are unique up to translations. When the Rabi coefficient is large, we prove that no non-constant solutions can exist.
\end{abstract}

\section{Introduction}
We study the monotonicity and uniqueness of a non-cooperative system coming from the physics of two-component Bose-Einstein condensates which displays partial phase transition. We are able to use the moving plane device in a case where a priori bounds do not come from standard techniques. The particularity of this system with respect to classical two-component segregated Bose-Einstein condensates is to couple both linearly and nonlinearly the two equations due to spin coupling of the hyperfine states added to the intercomponent coupling. Our system for the functions $u$ and $v$ is the following:
\begin{equation}\label{mainsystem}
\begin{cases}
\Delta u \, =- u(1-u^2-v^2) -v (\omega-\alpha uv) =: \RHSA(u,v)& \text{in}\quad\mathbb{R}^N,  \\
\Delta v \, = -v(1-u^2-v^2) -u (\omega-\alpha uv) =: \RHSB(u,v)& \text{in}\quad\mathbb{R}^N,
\end{cases}
\end{equation}
where $u$ and $v$ are the wave functions for each component, $\alpha$ is a positive parameter such that $1+\alpha$ is the intercomponent coupling, and $\omega$ is a positive parameter, which is the Rabi frequency of the one body coupling between the two internal states. We would like to study the phase transition solutions, so we will impose  boundary conditions at infinity in one direction, namely:
\begin{equation}\label{mainBC}
\begin{cases}
u(x',x_N)\to b,\quad  v(x',x_N)\to a, \quad& \text{as}\quad x_N\to +\infty,  \\
u(x',x_N)\to a,\quad  v(x',x_N)\to b, \quad& \text{as}\quad x_N\to -\infty,
\end{cases}
\end{equation} the limit being uniform in $x'\in \mathbb{R}^{N-1}$, where $(a,b)$ is the solution to
\begin{equation}\label{eqab}
a^2+b^2=1,\ \quad ab=\frac{\omega}\alpha,\quad {0 < a}< b.
\end{equation} The existence of such solutions requires, in addition to the positivity of $\alpha$ and $\omega$, the condition \begin{equation}\label{cond}
\frac{\omega}\alpha<\frac 12.\end{equation}

The study of domain wall solutions in coupled Gross-Pitaevskii equations or segregation patterns has been the subject of many papers concerning existence, uniqueness, monotonicity of asymptotic behaviour  \cite{aftsour17, alama, FSS,sourdisweak}. It corresponds to the case $\omega =0$ and $\alpha>0$.

Here, we would like to address a different physical background, that of a two-component Bose-Einstein condensates representing two different hyperfine states, and coupled by their spin, to take into account a one body coherent Rabi coupling, which corresponds to $ \omega >0$. This leads to what is called Rabi oscillations which have been experimentally observed in \cite{matexp}. The ground states and excited states have been studied in \cite{am,dror,sp,qustring,usui}. The  system (\ref{mainsystem})-(\ref{mainBC}) for $N=1$ has been analyzed in \cite{aftsour} where the existence and asymptotic properties of one dimensional domain wall solutions are derived in the case $\omega/\alpha$ of order 1 and $ \alpha$ large and small. The properties and structures found in \cite{aftsour} have led us to investigate the monotonicity and uniqueness of solutions. Let us point out that this system has a heteroclinic structure and derives from the minimization of the Gross-Pitaevskii energy.

Here is our main result:
\begin{theorem}\label{maintheo}
 Assume $\alpha>0$, $\omega > 0$, and (\ref{cond}) holds. Then all solutions to (\ref{mainsystem})-(\ref{mainBC}) depend only on $x_N$, satisfy $\partial_N u>0$ and $\partial_N v<0$ and are translations of one another along the $x_N$ direction.
 \end{theorem}

We note that the existence of solutions to (\ref{mainsystem})-(\ref{mainBC}) can be obtained by approximation by solutions on finite intervals as proved in Proposition \ref{Thm:EUMODE}.

The proof of our main Theorem relies on three main ingredients:
\begin{itemize}\item bounds which provide a special structure to the problem,
\item once these bounds are known, on the moving plane device, in a version which is inspired by \cite{FSS}. It is not a mere adaptation of previous works as it in fact requires new estimates.
    \item the uniqueness proof for one dimensional solutions, which is based once again on these bounds and on the sliding method \cite{BCN97}.
\end{itemize}
The key estimates hold for solutions of the system without any conditions at infinity. They are the following:
\begin{proposition}\label{propbounds}
Assume $\alpha>0$, $\omega > 0$, and (\ref{cond}) holds. Let $(u,v)$ be a positive regular solution of (\ref{mainsystem}), then \begin{equation}\label{keybound}u^2+v^2\leq 1\hbox{ and }uv \geq \frac\omega\alpha.
\end{equation} Moreover, either $(u,v) \equiv (a,b)$ or $(u,v) \equiv (b,a)$ or
 \begin{equation}
 a<u,v<b.\label{boundab}
 \end{equation}
\end{proposition}

Let us point out that once the bounds (\ref{keybound}) are known, the equations lead to  (\ref{boundab}), without a strict sign. Therefore, this provides signs to the right hand side of the system (\ref{mainsystem}), namely the system gets a structure of the type
 \begin{equation}\label{newsystem}
\begin{cases}
-\Delta u \, =u f_1(u,v)+v f_2(u,v) & \text{in}\quad\mathbb{R}^N,  \\
-\Delta v \, =v f_1(v,u)+u f_2(v,u)& \text{in}\quad\mathbb{R}^N,
\end{cases}
\end{equation}
with $f_1$ positive, symmetric, decreasing {in both variables} and $f_2$ negative, symmetric, decreasing {in both variables}. Whether the conclusion of Theorem \ref{maintheo} holds for more general systems of the form \eqref{newsystem} remains for further investigation.

 In the case $\omega=0$, then the solution to (\ref{eqab}) is given by $a = 0$ and $b=1$, and therefore domain wall solutions stay between $0$ and $1$. In the one-dimensional case the existence of a monotone solution has been derived in \cite{alama} while uniqueness has been derived in \cite{aftsour17}, based on a continuation argument and estimates of the linearized operator, under the assumption that either $u{'}>0$ or $v{'} <0$.
 Asymptotic estimates in the large $\alpha$ limit, still in one dimension, have been also obtained in \cite{aftsour17},  and rely on properties for a simpler outer system studied by many authors and for instance by \cite{beres1,beres2,FaSYM, FaSo}. The other limit, that of weak segregation where $\alpha$ tends to 0 has been analyzed in \cite{sourdisweak}. The proof that $N$-dimensional solutions are in fact one-dimensional and monotone is made in \cite{FSS} and relies on the moving plane device. The combination of the results of \cite{FSS} and \cite{aftsour17} then leads to the uniqueness of the solution, up to translations.

 Let us point out that, in the case $\omega=0$, the first bound in \eqref{keybound} is a quite direct consequence of the structure of the system  \eqref{mainsystem} (see \cite{FSS}, Theorem 1.3) while the second one is obvious for positive solutions. In our case, the situation is much more involved due to the presence of the Rabi frequency $\omega >0$, and proving these bounds requires some non-trivial extra-work, to which we devote the most part of Section 2.

We also get interesting results using similar estimates in the case where (\ref{cond}) does not hold.  We prove that solutions to \eqref{mainsystem} are constant without any boundary condition at infinity.
\begin{theorem}\label{maintheo2}
 Assume $\alpha>0$, and \begin{equation}\label{cond2}
\frac{\omega}\alpha \geq \frac 12\end{equation}
 holds. Let $(u,v)$ be a positive regular solution of (\ref{mainsystem}). Then $(u,v)$ is identically equal to $(c,c)$ with $c= \sqrt{\frac{1+\omega}{2+\alpha}}$.
 \end{theorem}

The paper is organized as follows: firstly, we prove our key a priori bounds, rough upper and lower bounds as well as the refined bounds of Proposition \ref{propbounds}. Then we make the moving plane device work to get monotonicity and the one dimensional property. Lastly, we study the properties of the one-dimensional solutions: existence, uniqueness up to translations and exponential decay to constant at infinity.

\section{A priori bounds. Proofs of Proposition \ref{propbounds}  and Theorem \ref{maintheo2}.}

The proofs of Proposition \ref{propbounds}  and Theorem \ref{maintheo2} follow from the same treatment. Note that the conclusion that $u \equiv v \equiv c$ in Theorem \ref{maintheo2} is equivalent to the pair of inequalities that $u^2 + v^2 \leq 2c^2$ and $uv \geq c^2$, which resembles \eqref{keybound}.

The rough idea of the proof is to bound the supremum of $u^2 + v^2$ in terms of the infimum of $uv$ and vice versa in such a way that the bounds self-bootstrap to the desired bounds. In the proof, we by-pass the fact that no boundary condition is imposed on $(u,v)$ at infinity by using Kato's inequality as in  \cite{Brezis, FSS}.

\subsection{A non-degeneracy estimate}

We start with a result which gives positive lower and upper bounds for $u + v$ for positive solutions to \eqref{mainsystem}.

\begin{lemma}\label{Lem:14V19-LB}
	Let $(\vone,\vtwo)$ be a positive regular solution of \eqref{mainsystem}. Then
	\[
	\Big[ \frac{ 1 + \omega}{\max (\frac{2 + \alpha}{4}, 1)}\Big]^{1/2} \leq
	\vone + \vtwo \leq \Big[ \frac{ 1 + \omega}{\min (\frac{2 + \alpha}{4}, 1)}\Big]^{1/2} \text{ in } \RR^n.
	\]
	In particular $\vone, \vtwo \in L^{\infty}(\RR^n).$
\end{lemma}

\begin{proof}
	Let $w = \vone + \vtwo$. Then
	\begin{equation}
	\Delta w = w(\vone^2 + \vtwo^2 + \alpha\vone\vtwo - 1 - \omega).
	\label{Eq:14V18-E1}
	\end{equation}
	Thus, by the strong maximum principle, we have $w > 0$ in $\RR^n$.
	
	In view of \eqref{Eq:14V18-E1}, we have
	\begin{equation}
	\Delta w \geq w\Big( \min (\frac{2 + \alpha}{4}, 1)  w^2- 1 - \omega\Big) = \min (\frac{2 + \alpha}{4}, 1) w (w^2 - \gamma^2)
	\label{Eq:14V18-E2}
	\end{equation}
	where $\gamma =\Big[ \frac{ 1 + \omega}{\min (\frac{2 + \alpha}{4}, 1)}\Big]^{1/2}$.
	
	By Kato's inequality, this implies that
	\[
	\Delta (w - \gamma)_+ \geq  \min (\frac{2 + \alpha}{4}, 1)  (w - \gamma)_+^3
	\]
	which further implies $(w - \gamma)_+ \equiv 0$ (see \cite[Lemma 2]{Brezis}), i.e.
	\begin{equation}
	w \leq \gamma \text{ in } \RR^n .
	\label{Eq:14V18-E3}
	\end{equation}
	Returning to \eqref{Eq:14V18-E1}, we see that
	\begin{equation}
	\Delta w \leq w\Big( \max (\frac{2 + \alpha}{4}, 1)  w^2- 1 - \omega\Big) .
	\label{Eq:14V18-E5}
	\end{equation}
	and thus the rescaled function  $\tilde w(x) = \max (\frac{(2 + \alpha)^{1/2}}{2}, 1) (1 + \omega)^{-1/2}w((1 + \omega)^{-1/2}x)$
	is a positive supersolution of the Allen-Cahn equation, i.e., it satisfies
	
	\begin{equation}
	-\Delta \tilde w \geq \tilde w\Big( 1 - \tilde w^2 \Big) \quad \text{ in }\,\, \RR^n.
	\label{Eq:14V18-E6}
	\end{equation}

	To proceed, we need the following lemma.
	
	\begin{lemma}\label{Lem:14V19-MM}
		There exists $R_0 > 0$ such that, for $ R\geq R_0$, the functional
		\[
		I[\varphi] = \int_{B_R} \Big[\frac{1}{2} |\nabla \varphi|^2 + \frac{1}{4}(1 - \varphi^2)^2\Big]\,dx
		\]
		admits a non-trivial minimizer $\psi$ in $H_0^1(B_R)$. Furthermore,  $\psi$ is a smooth function in ${\overline {B_R}}$  satisfying $ 0 < \psi < 1$ in $ B_R$.
	\end{lemma}
	
	Let us assume the above lemma for the moment and continue with the proof of Lemma \ref{Lem:14V19-LB}. Let $R_0$ be the constant in Lemma \ref{Lem:14V19-MM} and $ m := \min_{{\overline {B_{R_0}}}} {\tilde w} >0$ then, for every $\varepsilon \in (0, \min \{m,1\}) $,  the function $\psi_{\varepsilon} := \varepsilon \psi$ satisfies

	\[
	-\Delta \psi_{\varepsilon}  \leq \psi_{\varepsilon} (1 - \psi_{\varepsilon}^2) \quad\text{ in }\,\, \RR^n,
	\]
	and
	\[
	\psi_{\varepsilon}  \leq \tilde w \quad\text{ in } \,\,B_{R_0}, \qquad \psi_{\varepsilon} = 0 \quad\text{ on }\,\, \partial B_{R_0}.
	\]
	The sliding method (see \cite[Lemma 3.1]{BCN97}), then gives that $\tilde w\geq \varepsilon$ in $ \RR^n$ and so $ w \geq \delta $ in $ \RR^n$, for some $\delta>0$.
	
	Now, from \eqref{Eq:14V18-E5} we get
	
	\begin{equation}
	- \Delta w \geq w\Big( 1+\omega - \max (\frac{2 + \alpha}{4}, 1)  w^2\Big) \geq \delta \max(\frac{2 + \alpha}{4}, 1) (\theta^2 - w^2)
	\label{Eq:14V18-E7}
	\end{equation}
	where $\theta =\Big[ \frac{ 1 + \omega}{\max (\frac{2 + \alpha}{4}, 1)}\Big]^{1/2}$. Hence, by Kato's inequality we have
	
	\begin{equation}
	\Delta (\theta - w)_+ \geq  \delta \max(\frac{2 + \alpha}{4}, 1) [(\theta - w)_+]^2
	\label{Eq:14V18-E8}
	\end{equation}
	which implies $(\theta - w)_+ \equiv 0$ (see \cite[Lemma 2]{Brezis}), i.e.
	\begin{equation}
	w \geq \Big[ \frac{ 1 + \omega}{\max (\frac{2 + \alpha}{4}, 1)}\Big]^{1/2}\text{ in } \RR^n .
	\label{Eq:14V18-E9}
	\end{equation}
	
\end{proof}

\begin{proof}[Proof of Lemma \ref{Lem:14V19-MM}]
We have that $I[0] = \frac{1}{4} |B_R|$. Suppose that $R > 1$ and consider the function
\[
\varphi(x) = \min(R - |x| , 1) \text{ for } x \in B_R.
\]
We have
\[
I[\varphi] =  \int_{B_R \setminus B_{R-1}} \Big[\frac{1}{2} |\nabla \varphi|^2 + \frac{1}{4}(1 - \varphi^2)^2\Big]\,dx \leq \frac{3}{4} |B_R \setminus B_{R-1}|.
\]
Clearly, for $R$ sufficiently large, $I[\varphi] < I[0]$ and so $I$ posseses a non-trivial minimizer $\psi$ in $ H^1_0(B_R)$. Replacing $\psi$ by $ \min \{ \vert \psi \vert, 1\}$, if necessary, we may assume that $ 0 \le \psi \le 1$ in ${\overline {B_R}}$. Therefore $ \psi$ is a weak solution of the Allen-Cahn equation in  ${\overline {B_R}}$ and the remaining part of the claim follows by standard elliptic regularity and by the strong maximum principle.
\end{proof}

%+++++%

\subsection{Proof of Proposition \ref{propbounds}}
\begin{proof} We have $u, v > 0$, and, by Lemma \ref{Lem:14V19-LB}, there is some $p > 1$ such that
\begin{equation}
\vone + \vtwo \geq \frac{1}{p}.
\label{Eq:StrPositivity}
\end{equation}

Let $A = \vone^2 + \vtwo^2$, $B = -\ln(\vone \vtwo)$ and
\begin{equation}
m = \sup A, \quad m_* = \max(m, 1), \quad n = \sup B,  \quad \text{ and } \quad n_* = \max(n, -\ln \frac{\omega}{\alpha}).
	\label{Eq:26VI19-R1}
\end{equation}

\medskip
\noindent 1. We prove that
\begin{equation}
m \leq \frac{1}{2} (1 + \sqrt{1 - 8s_*}) \text{ where } s_* = \min_{t \geq e^{-n_*}} (\alpha t^2 - \omega t).
	\label{Eq:15IV19-R1}
\end{equation}
In particular, as $s_* \geq \frac{-\omega^2}{4\alpha}$, we have
\begin{equation}
m \leq \frac{1}{2} \Big(1 + \sqrt{1 + \frac{2\omega^2}{\alpha}}\Big) < \infty.
	\label{Eq:16IV19-M1}
\end{equation}

We have
\begin{align*}
\Delta A
	&\geq 2\vone \RHSA(\vone, \vtwo) + 2\vtwo \RHSB(\vone, \vtwo) = 2\vone^4 + 2\vtwo^4 + 4(\alpha + 1)\vone^2 \vtwo^2 - 2\vone^2 - 2\vtwo^2 - 4\omega \vone \vtwo\\
	&= 2A^2 - 2A + 4\alpha e^{-2B} - 4\omega e^{-B}.
\end{align*}
Note that $\alpha e^{-2B} - \omega e^{-B} \geq s_*$, and, as $n_* \geq - \ln \frac{\omega}{\alpha}$, $s_* \leq 0$. It follows that
\[
\Delta A
	\geq 2A^2 - 2A + 4s_* = 2\Big(A - \frac{1}{2} (1 + \sqrt{1 - 8s_*})\Big) \Big(A - \frac{1}{2} (1 - \sqrt{1 - 8s_*})\Big),
\]
and so, by Kato's inequality,
\[
\Delta  \Big(A - \frac{1}{2} (1 + \sqrt{1 - 8s_*})\Big)_+
	\geq 2A^2 - 2A + 4s_* = 2\Big(A - \frac{1}{2} (1 + \sqrt{1 - 8s_*})\Big)_+^2.
\]
Estimate \eqref{Eq:15IV19-R1} follows from \cite[Lemma 2]{Brezis}.

\medskip
\noindent 2. We prove that
\begin{equation}
n \leq - \ln \frac{\omega}{\alpha + 2 - \frac{2}{m_*}}.
	\label{Eq:15IV19-R2}
\end{equation}
Equivalently,
\begin{equation}
\vone\vtwo \geq \frac{\omega}{\alpha + 2 - \frac{2}{m_*}}.
	\label{Eq:15IV19-R1al}
\end{equation}

We have
\begin{align*}
\Delta B
	&\geq -\frac{1}{\vone} \RHSA(\vone,\vtwo) - \frac{1}{\vtwo} \RHSB(\vone, \vtwo)
		= -(\alpha + 2)(\vone^2 + \vtwo^2) + 2 + \frac{\omega(\vone^2 + \vtwo^2)}{\vone\vtwo}\\
	&= -(\alpha + 2) A + 2 + \omega A e^{B} = 2(1 - \frac{A}{m_*}) + (\alpha + 2 - \frac{2}{m_*}) A(\frac{\omega}{\alpha + 2 - \frac{2}{m_*}}e^B - 1).
\end{align*}
Using Kato's inequality, the inequality $e^x - 1 \geq \frac{1}{2}x^2$ for $x \geq 0$ and recalling \eqref{Eq:StrPositivity}, we get
\begin{align*}
\Delta (B - \ln \frac{\alpha + 2 - \frac{2}{m_*}}{\omega})_+
	&\geq \frac{\alpha + 2 - \frac{2}{m_*}}{4p^2}  (B - \ln \frac{\alpha + 2 - \frac{2}{m_*}}{\omega})_+^2.
\end{align*}
 We deduce that $(B - \ln \frac{\alpha + 2 - \frac{2}{m_*}}{\omega})_+ \equiv 0$, again thanks to \cite[Lemma 2]{Brezis}. This proves \eqref{Eq:15IV19-R2}.

\medskip
\noindent 3. We prove that $m_* = 1$ or $n_* = -\ln \frac{\omega}{\alpha}$.

Assume by contradiction that the above does not hold. Then $m = m_* > 1$ and $n = n_* > -\ln \frac{\omega}{\alpha}$.

Let
\begin{align}
h_1(t)
	&= \frac{1}{2}\Big(1 + \sqrt{1 - 8\alpha t^2 + 8\omega t}\Big),\label{Eq:19h1Def}\\
h_2(t)
	&= \frac{\omega t}{(\alpha + 2)t - 2}.\label{Eq:19h2Def}
\end{align}
We claim that
\begin{equation}
m \leq h_1(h_2(m)).
	\label{Eq:15IV19-R3}
\end{equation}

\medskip
\noindent\underline{Case (i)}: $\alpha \geq 2$. From \eqref{Eq:15IV19-R2}, we have $n \leq - \ln \frac{\omega}{\alpha + 2} \leq -\ln \frac{\omega}{2\alpha}$. It follows that
\[
s_* = \min_{t \geq e^{-n}} (\alpha t^2 - \omega t) = \alpha e^{-2n} - \omega e^{-n}.
\]
Plugging this into \eqref{Eq:15IV19-R1} yields
\[
m \leq \frac{1}{2}(1 + \sqrt{1 - 8\alpha e^{-2n} + 8\omega e^{-n}}) = h_1(e^{-n}).
\]
Since $h_1$ is decreasing in $[\frac{\omega}{2\alpha},\infty)$, this together with \eqref{Eq:15IV19-R2} implies \eqref{Eq:15IV19-R3}.

\medskip
\noindent\underline{Case (ii)}: $\alpha < 2$. As $\alpha < 2$, $\omega < \frac{\alpha}{2} < \frac{2\alpha}{2 - \alpha}$ and so, by \eqref{Eq:16IV19-M1},
\[
m \leq \frac{1}{2}\Big(1 + \sqrt{1 + \frac{2\omega^2}{\alpha}}\Big) < \frac{1}{2}\Big(1 + \sqrt{1 + \frac{8\alpha}{(2-\alpha)^2}}\Big) = \frac{2}{2 - \alpha}.
\]
Inserting this into \eqref{Eq:15IV19-R2} yields $n \leq - \ln \frac{\omega}{2\alpha}$. We can now repeat the proof of Case (i) to reach \eqref{Eq:15IV19-R3}.

We now compute
\begin{align}
\frac{d}{dt} h_1(h_2(t))
	&= h_1'(h_2(t)) h_2'(t) = \frac{2(2\alpha h_2(t) - \omega)}{\sqrt{1 - 8\alpha h_2(t)^2 + 8\omega h_2(t)}}  \frac{2\omega}{[(\alpha + 2)t - 2]^2}\nonumber\\
	&= \frac{1}{\sqrt{1 - 8\alpha h_2(t)^2 + 8\omega h_2(t)}}  \frac{4\omega^2(\alpha t - 2 t + 2)}{(\alpha t + 2t - 2)^3}.
		\label{Eq:19IV19-T4}
\end{align}
Note that $h_2$ is decreasing in $(\frac{2}{2 + \alpha},\infty)$ and so $h_2(t) < \frac{\omega}{\alpha}$ for $t > 1$. Hence, for $t > 1$, we have $\sqrt{1 - 8\alpha h_2(t)^2 + 8\omega h_2(t)} > 1$, $\alpha t - 2 t + 2 < \alpha t + 2 t - 2$ and $4\omega^2 < \alpha^2 < (\alpha t + 2t - 2)^2$ and so
\[
\frac{d}{dt} h_1(h_2(t)) < 1 \text{ for all } t > 1.
\]
As $h_1(h_2(1)) = 1$, this implies that the inequality equation $t \leq h_1(h_2(t))$ has no solution in $(1,\infty)$. Therefore, \eqref{Eq:15IV19-R3}  implies that $m=1$. This finishes Step 3.

\medskip
\noindent 4. We prove \eqref{keybound}. By Step 3, we have $m_* = 1$ or $n_* = -\ln \frac{\omega}{\alpha}$.

If $m_* = 1$, then, in view of \eqref{Eq:15IV19-R2}, $n \leq -\ln \frac{\omega}{\alpha}$, and so $A \leq 1$ and $B \geq -\ln \frac{\omega}{\alpha}$, which give \eqref{keybound}.

On the other hand, if $n_* = -\ln \frac{\omega}{\alpha}$, then, by \eqref{Eq:15IV19-R1}, $m \leq 1$. Again we obtain $A \leq 1$ and $B \geq -\ln \frac{\omega}{\alpha}$, which also give \eqref{keybound} as desired.

\medskip
\noindent 5. Finally, we prove the trichotomy that either $(u,v) \equiv (a,b)$ or $(u,v) \equiv (b,a)$ or \eqref{boundab} holds.

Suppose that $(u,v) \not\equiv (a,b)$ and $(u,v) \not\equiv (b,a)$. From \eqref{keybound}, we have
\begin{equation}
a \leq \vone, \vtwo \leq b.
	\label{Eq:31I19-M5}
\end{equation}

We note that
\[
\RHSA(b,\vtwo)
	= b(b^2 + \vtwo^2 - 1) + \vtwo(\alpha b \vtwo - \omega) \stackrel{\eqref{Eq:31I19-M5}}{\geq } b(b^2 + a^2 - 1) + \vtwo(\alpha b a - \omega) = 0.
\]
Hence the constant function $b$ satisfies
\[
\Delta b = 0 \leq \RHSA(b,\vtwo).
\]
Since $\vone \leq b$, and $\Delta\vone = \RHSA(\vone, \vtwo)$, the strong maximum principle implies either $\vone < b$ or $\vone \equiv b$. If the latter case holds, the second equation of \eqref{mainsystem} implies that $\vtwo \equiv a$, which contradicts our assumption that $(\vone,\vtwo)\not\equiv (b,a)$. We thus have $\vone < b$.

The remaining inequalites in \eqref{boundab} are shown similarly using
\begin{align*}
\RHSA(a,\vtwo)
	&= a(a^2 + \vtwo^2 - 1) + \vtwo(\alpha a \vtwo - \omega) \stackrel{\eqref{Eq:31I19-M5}}{\leq } a(a^2 + b^2 - 1) + \vtwo(\alpha a b - \omega) = 0,\\
\RHSB(\vone,a)
	&= a( \vone^2 + a^2 - 1) + \vone(\alpha \vone a - \omega) \stackrel{\eqref{Eq:31I19-M5}}{\leq } a(b^2 + a_2^2 - 1) + \vone(\alpha b a - \omega) = 0,\\
\RHSB(\vone, b)
	&= b( \vone^2 + b^2 - 1) + \vone(\alpha \vone b - \omega) \stackrel{\eqref{Eq:31I19-M5}}{\geq } b(a^2 + a_2^2 - 1) + \vone(\alpha  a b - \omega) = 0.
\end{align*}
We omit the details.
\end{proof}

\subsection{Proof of Theorem \ref{maintheo2}}

\begin{proof}
We adapt the proof of Proposition \ref{propbounds}, as the conclusion is equivalent to the following pair of inequalities:
\begin{align*}
\vone^2 + \vtwo^2 \leq \frac{2(1 + \omega)}{2+\alpha} \text{ and } \vone \vtwo \geq \frac{1 + \omega}{2+\alpha}.
\end{align*}

By Lemma \ref{Lem:14V19-LB}, \eqref{Eq:StrPositivity} holds. Let $A = \vone^2 + \vtwo^2$, $B = -\ln(\vone \vtwo)$ and
\begin{align*}
m
	&= \sup A, \quad \tilde m_* = \max(m, \frac{2(1 + \omega)}{2+\alpha}), \\
n
	&= \sup B,  \quad \text{ and } \quad \tilde n_* = \max(n, -\ln \frac{1 + \omega}{2 + \alpha}).
\end{align*}
(Note the difference between the definition of $\tilde m_*$ and $\tilde n_*$ and that of $m_*$ and $n_*$ in the proof of Proposition \ref{propbounds}.)

1. We prove that
\begin{equation}
m \leq \frac{1}{2} \min\Big(1 + \sqrt{1 - 8s_1}, 1 + \omega + \sqrt{(1 + \omega)^2 - 8\alpha s_2}\Big),
	\label{Eq:19IV19-T1}
\end{equation}
where $s_1 = \min_{t \geq e^{-\tilde n_*}} (\alpha t^2 - \omega t) \leq 0$ and $s_2 = e^{-2\tilde  n_*} \leq \frac{(1 + \omega)^2}{8\alpha}$.

As $\tilde n_* \geq -\ln \frac{1 + \omega}{2 + \alpha} \geq - \ln \frac{\omega}{\alpha}$ (due to $\omega \geq \frac{1}{2}\alpha$), $s_1 \leq 0$. Also $s_2 \leq \frac{(1 + \omega)^2}{(2 + \alpha)^2} \leq \frac{(1 + \omega)^2}{8\alpha}$.

The proof of the inequality $m \leq \frac{1}{2} (1 + \sqrt{1 - 8s_1})$ follows from the differential inequality
\begin{align*}
\Delta A
	&\geq 2A^2 - 2A + 4\alpha \vone^2 \vtwo^2 - 4\omega \vone \vtwo,
\end{align*}
exactly as in the proof of \eqref{Eq:15IV19-R1}. To obtain $m \leq \frac{1}{2}(1 + \omega + \sqrt{(1 + \omega)^2 - 8\alpha s_2})$, we use the inequality $v_1v_2 \leq \frac{1}{2}A$ in the above differential inequality:
\begin{align*}
\Delta A
	&\geq 2A^2 - 2(1 + \omega) A + 4\alpha e^{-2B}
		\geq 2A^2 - 2(1 + \omega) A + 4\alpha s_2\\
	&= 2\Big(A - \frac{1}{2}(1 + \omega + \sqrt{(1 + \omega)^2 - 8\alpha s_2})\Big)\Big(A - \frac{1}{2}(1 + \omega - \sqrt{(1 + \omega)^2 - 8\alpha s_2})\Big).
\end{align*}
By Kato's inequality, this leads to
\[
\Delta \Big(A - \frac{1}{2}(1 + \omega + \sqrt{(1 + \omega)^2 - 8\alpha s_2})\Big)_+
	\geq 2 \Big(A - \frac{1}{2}(1 + \omega + \sqrt{(1 + \omega)^2 - 8\alpha s_2})\Big)_+^2
\]
and so, by \cite[Lemma 2]{Brezis},
\[
A \leq \frac{1}{2}(1 + \omega + \sqrt{(1 + \omega)^2 - 8\alpha s_2}).
\]
We have thus proved \eqref{Eq:19IV19-T1}.

\medskip
\noindent 2. As in the proof of Proposition \ref{propbounds}, we have
\begin{equation}
n \leq - \ln \frac{\omega}{\alpha + 2 - \frac{2}{\tilde m_*}}.
	\label{Eq:19IV19-Tp2}
\end{equation}

\medskip
\noindent
3. We show that $\tilde m_* = \frac{2(1 + \omega)}{2+\alpha}$ or $\tilde n_* = -\ln \frac{1 + \omega}{2+\alpha}$. Assume by contradiction that this does not hold, so that $m = \tilde m_* > \frac{2(1 + \omega)}{2+\alpha}$ and $n = \tilde n_* > -\ln \frac{1 + \omega}{2+\alpha}$.

\medskip\noindent\underline{Case (a):} $\frac{1 + \omega}{2 + \alpha} \geq \frac{\omega}{2\alpha}$ (i.e. either $\alpha \geq 2$ or $0 < \alpha < 2$ and $\omega \leq \frac{2\alpha}{2 - \alpha}$).
\medskip

In this case, the argument in Step 3 of the proof of Proposition \ref{propbounds} gives
\begin{equation}
m \leq h_1(h_2(m)),
	\label{Eq:19IV19-T2}
\end{equation}
where $h_1$ and $h_2$ are defined in \eqref{Eq:19h1Def}-\eqref{Eq:19h2Def}

Now note that $h_1(h_2(\frac{2(1+ \omega)}{2 + \alpha} )) = \frac{2(1+ \omega)}{2 + \alpha}$. Thus in order to obtain a contradiction, it suffices to show that
\begin{equation}
\frac{d}{dt} h_1(h_2(t))  < 1 \text{ for all } t >  \frac{2(1+ \omega)}{2 + \alpha} .
	\label{Eq:19IV19-T3}
\end{equation}
To see this, recall formula \eqref{Eq:19IV19-T4} for the derivative of $h_1 \circ h_2$:
\begin{align*}
\frac{d}{dt} h_1(h_2(t))
	&= \frac{1}{\sqrt{1 - 8\alpha h_2(t)^2 + 8\omega h_2(t)}}  \frac{4\omega^2(\alpha t - 2 t + 2)}{[(\alpha + 2)t - 2]^3}.
\end{align*}
Now if $t > \frac{2(1+\omega)}{2 + \alpha}$, then as $h_2$ is decreasing in $(\frac{2}{2 + \alpha}, \infty)$, we have $\frac{\omega}{2 + \alpha} = h_2(\infty) < h_2(t) < h_2( \frac{2(1+\omega)}{2 + \alpha}) = \frac{1 + \omega}{2 + \alpha} < \frac{\omega}{\alpha}$ (thanks to $\alpha < 2\omega$), and so
\[
1 - 8\alpha h_2(t)^2 + 8\omega h_2(t) > 1.
\]
Also, for $t > \frac{2(1+\omega)}{2 + \alpha} > 1$, we have $\alpha t - 2 t + 2 < \alpha t + 2 t - 2$ and $4\omega^2 < (\alpha t + 2t - 2)^2$. \eqref{Eq:19IV19-T3} hence follows. This concludes Case (a).

\medskip\noindent\underline{Case (b):} $\frac{1 + \omega}{2 + \alpha} < \frac{\omega}{2\alpha}$ (i.e. $0 < \alpha < 2$ and $\omega > \frac{2\alpha}{2 - \alpha}$).
\medskip

We start by showing that
\begin{equation}
m \leq \tilde h_1(h_2(m)),
	\label{Eq:19IV19-A1}
\end{equation}
where $h_2$ is defined in \eqref{Eq:19h2Def} and $\tilde h_1$ is defined by
\[
\tilde h_1(t) = \frac{1}{2}(1 + \omega + \sqrt{(1 + \omega)^2 - 8\alpha t^2}).
\]
Indeed, By \eqref{Eq:19IV19-Tp2}, $n \leq -\ln h_2(m)$. Plugging this into \eqref{Eq:19IV19-T1}, we get $m \leq \tilde h_1(e^{-n}) \leq \tilde h_1(h_2(m))$, as $\tilde h_1$ is decreasing in $(0,\infty)$.

Next, a direct computation gives $\tilde h_1(h_2(\frac{2(1+ \omega)}{2 + \alpha} )) = \frac{2(1+ \omega)}{2 + \alpha}$. Thus, as in Case (a), it suffices to show that
\begin{equation}
\frac{d}{dt} \tilde h_1(h_2(t))  < 1 \text{ for all } t >  \frac{2(1+ \omega)}{2 + \alpha} .
	\label{Eq:19IV19-A2}
\end{equation}
We compute
\begin{align*}
\frac{d}{dt} h_1(h_2(t))
	&= h_1'(h_2(t)) h_2'(t) = \frac{4\alpha h_2(t)}{\sqrt{(1 + \omega)^2 - 8\alpha h_2(t)^2}}  \frac{2\omega}{[(\alpha + 2)t - 2]^2}\\
	&= \frac{1}{\sqrt{(1 + \omega)^2 - 8\alpha h_2(t)^2}}  \frac{8\alpha \omega h_2(t)}{[(\alpha + 2)t - 2]^2}.
\end{align*}
Now for $t > \frac{2(1+\omega)}{2 + \alpha}$, we have $\frac{\omega}{2 + \alpha} < h_2(t) < \frac{1 + \omega}{2 + \alpha}$ as in the previous case thanks to the monotonicity of $h_2$. It follows that
\begin{align*}
\frac{d}{dt} h_1(h_2(t))
	&< \frac{1}{\sqrt{(1 + \omega)^2 - 8\alpha \frac{(1 + \omega)^2}{(2 + \alpha)^2}}}  \frac{8\alpha \omega \frac{1 + \omega}{2 + \alpha}}{4\omega^2} = \frac{2\alpha}{(2 - \alpha)\omega} < 1 \text{ for all } t > \frac{2(1+\omega)}{2 + \alpha}.
\end{align*}
This proves \eqref{Eq:19IV19-A2}, and so finishes Case (b). Step 3 is concluded.

\medskip
\noindent
4. Finally, we show that $\vone \equiv \vtwo \equiv \sqrt{\frac{1 + \omega}{2 + \alpha}}$.

By Step 3, we have $\tilde m_* = \frac{2(1+\omega)}{2 + \alpha}$ or $\tilde n_* = -\ln \frac{1+\omega}{2 + \alpha}$.

If $\tilde m_* = \frac{2(1+\omega)}{2 + \alpha}$, then, in view of \eqref{Eq:15IV19-R2}, $n \leq -\ln \frac{1+\omega}{2 + \alpha}$, and so $A \leq \frac{2(1+\omega)}{2 + \alpha}$ and $B \geq -\ln\frac{1+\omega}{2 + \alpha}$, which give $\vone \equiv \vtwo \equiv \sqrt{\frac{1 + \omega}{2 + \alpha}}$.

On the other hand, if $\tilde n_* = -\ln \frac{1+\omega}{2 + \alpha}$, then, by \eqref{Eq:15IV19-R1}, $m \leq  \frac{2(1+\omega)}{2 + \alpha}$. Again we obtain $A \leq \frac{2(1+\omega)}{2 + \alpha}$ and $B \geq -\ln\frac{1+\omega}{2 + \alpha}$, which then give $\vone \equiv \vtwo \equiv \sqrt{\frac{1 + \omega}{2 + \alpha}}$. We conclude the proof.
\end{proof}

%----------------------------------------------------------------------------%
\section{Moving plane device}\label{Sec:MovingPlanes}
\subsection{Monotonicity with respect to $x_N$}\label{mono}

We are going to show that if $(u,v)$ is a solution to (\ref{mainsystem})-(\ref{mainBC}), then it is monotone with respect to $x_N$. This relies strongly on the estimates
  (\ref{keybound})-(\ref{boundab}).

\begin{prop}\label{propmonotxN}
Under the assumptions of Theorem \ref{maintheo}, since (\ref{keybound})-(\ref{boundab}) hold, we have
\begin{equation}\label{monotxN}
\pa_N u>0 \quad \text{and} \quad \pa_N v<0 \qquad \text{in $\R^N$}.
\end{equation}
\end{prop}
The proof is based on  the moving planes method, in a version developed by \cite{FSS}. We follow \cite{FSS}, nevertheless our system requires some major adjustments as we will point out.

For $\lambda \in \R$, we set
\[
u_{\lambda}(x',x_N):=u(x',2\lambda - x_N), \; v_{\lambda}(x',x_N):=v(x',2\lambda - x_N) \quad \text{and} \quad \Sigma_\la:= \{x_N > \la\}.
\]
We aim at proving that
\begin{equation}\label{moving thesis 0}
u_\la(x) \le u(x) \quad \text{and} \quad v_\la(x) \ge v(x) \quad \forall x \in \Sigma_\la, \ \forall \la \in \R.
\end{equation}
This and the strong Maximum Principle will yield Proposition \ref{propmonotxN}.

In order to prove that \eqref{moving thesis 0} is satisfied, we show that
\begin{equation}\label{moving thesis 1}
\Lambda:=\left\{ \la \in \R: \text{$u_\mu \le u$ and $v_\mu \ge v$ in $\Sigma_\mu$ for every $\mu \ge \la$}\right\}= \R.
\end{equation}

We will make great use of a lemma proved and used in \cite{FMS} to show that the positive part or negative part of some functions are identically zero.
\begin{lem}[{\cite[Lemma 2.1]{FMS}}]\label{lemLR}
Let $\theta >0$ and $\gamma>0$ such that $\theta < 2^{-\gamma}$.
Moreover let $R_0>0$, $C>0$ and
$$\mathcal{L}:(R_0, + \infty) \rightarrow \mathbb{R}$$ be a non-negative and non-decreasing function such that
\[
\begin{cases} \mathcal{L}(R)\leq \theta \mathcal{L}(2R)+G(R) & \forall R>R_0,\\ \mathcal{L}(R)\leq CR^{\gamma} & \forall R >R_0, \end{cases}
\]
where $G:(R_0, +\infty)\rightarrow \mathbb{R}^+$ is such that $$\lim_{R\rightarrow +\infty}G(R)=0 .$$ Then $$\mathcal{L}(R)=0 \qquad \forall R > R_0.$$
\end{lem} We can start the moving plane device and  prove that $\Lambda \neq \emptyset$:
\begin{lem}\label{monotinf}
There exists $\bar \lambda \in \R$ sufficiently large such that
\[
u\geq u_\lambda \quad\text{and}\quad v\leq v_\lambda \qquad \text{in}\,\,\Sigma_\lambda
\]
for any $\lambda\geq\bar\lambda$.
In other words, $\Lambda \supset [\bar\lambda,\infty)$.
\end{lem}
\begin{proof}
The pair $(u_\lambda,v_\lambda)$ solves
\begin{equation}\tag{$P_\lambda$}\label{syst lambda}
\begin{cases}
-\Delta u_\lambda
	=g(u_\lambda,v_\lambda)+v_\lambda (\omega-\alpha u_\lambda v_\lambda) \\
-\Delta v_\lambda
	= g(v_\lambda,u_\lambda)+u_\lambda (\omega-\alpha u_\lambda v_\lambda) \\
a<u_\lambda,v_\lambda<b
\end{cases}
\end{equation} where $g(u,v)=u(1-u^2-v^2)$ and where we have used \eqref{boundab}.

Let $\varphi_R$ be a standard $\mathcal{C}^1$ cut-off function on $\RR^N$ such that $\varphi_R=1$ in $B_R$, $\varphi_R=0$ outside $B_{2R}$ and $|\nabla\varphi_R|\leq 2/R$ on $\RR^N$. We subtract the equations for $u_\lambda$ and $u$, and multiply by the  test function
\begin{equation}\nonumber
(u_\lambda-u)^+\varphi_R^2 \, \mathds{1}_{ \Sigma_\lambda}.
\end{equation}
We find
\begin{equation}\label{estu}\int_{\Sigma_\lambda}\,|\nabla (u_\lambda-u)^+|^2\varphi_R^2\,=\,-2\int_{\Sigma_\lambda} \varphi_R (u_\lambda-u)^+ \, \nabla (u_\lambda-u)^+ \cdot \nabla\varphi_R+I_1+I_2,\end{equation}
where
\begin{equation}\label{I1}
I_1=\int_{\Sigma_\lambda}\big(g(u_\lambda,v_\lambda)-g(u,v)\big)(u_\lambda-u)^+\varphi_R^2,
\end{equation}
\begin{equation}\label{I2}
I_2=\int_{\Sigma_\lambda}\big(v_\lambda(\omega -\alpha u_\lambda v_\lambda)-v(\omega -\alpha uv) \big)(u_\lambda-u)^+\varphi_R^2\,.
\end{equation}
We proceed similarly by subtracting the equations for $v$ and $v_\lambda$  and multiplying by
$$
(v-v_\lambda)^+\varphi_R^2\, \mathds{1}_{ \Sigma_\lambda}\,
$$
and get
\begin{equation}\label{estv}
\int_{\Sigma_\lambda}\,|\nabla (v-v_\lambda)^+|^2\varphi_R^2\,=\,-2\int_{\Sigma_\lambda} \varphi_R (v-v_\lambda)^+ \, \nabla (v-v_\lambda)^+\cdot \nabla\varphi_R+I_3+I_4\end{equation}
where
\begin{equation}\label{I3}
I_3=\int_{\Sigma_\lambda}\big(g(v,u)-g(v_\lambda,u_\lambda)\big)(v-v_\lambda)^+\varphi_R^2,
\end{equation}
\begin{equation}\label{I4}
I_4=\int_{\Sigma_\lambda}\big( u (\omega - \alpha uv) - u_\lambda(\omega - \alpha u_\lambda v_\lambda) \big)(v-v_\lambda)^+\varphi_R^2.
\end{equation}
Let
\begin{align*}
{\mathcal L}_\lambda(R)\,
	&:=\, \int_{\Sigma_\lambda\cap B_R}\,\Big[|\nabla (u_\lambda-u)^+|^2\,+\,|\nabla (v-v_\lambda)^+|^2\Big]
,\\
\mathcal{J}_\lambda(R)\,
	&:=\, \int_{\Sigma_\lambda}\,\Big[\big( (u_\lambda-u)^+\big)^2+\big( (v-v_\lambda)^+\big)^2 \Big] {\varphi_R^2}.
\end{align*}
We deduce from \eqref{estu}-\eqref{estv} that for any $\theta \in (0,1)$,
\begin{equation}
{\mathcal L}_\lambda (R)\leq \theta {\mathcal L}_\lambda (2R)
+\frac{4}{\theta R^2} {\mathcal J}_\lambda(R)+I_1+I_2+I_3+I_4.
	\label{Eq:LR2R1234}
\end{equation}
Therefore, we will need to estimate $I_1$, $I_2$, $I_3$, $I_4$ in terms of
${\mathcal J}_\lambda(R)$ in order to be able to use  Lemma \ref{lemLR} and deduce that $\mathcal L_\lambda (R)\equiv 0$, which will imply the conclusion of the Lemma.

\noindent{\em Estimate of $I_2$ and $I_4$:} We deduce  from (\ref{I2}) that
\begin{equation}\label{I2new}
I_2=\int_{\Sigma_\lambda}\Big[(v- v_\lambda)(\alpha u (v+ v_\lambda)-\omega)(u_\lambda-u)^+ - \alpha v_\lambda^2\big ((u_\lambda-u)^+\big )^2 \Big]\varphi_R^2\,.
\end{equation}
We recall from Proposition \ref{propbounds} that $\alpha uv -\omega \geq 0$. So that in the first term in the square bracket on the right hand side of of \eqref{I2new}, we can keep only $(v-v_\lambda)^+$ in the upper bound and find
\begin{equation}\label{I2new2}
I_2\leq\int_{\Sigma_\lambda } \Big[(\alpha u v-\omega)(v-v_\lambda)^+(u_\lambda-u)^+ +\alpha u v_\lambda (v-v_\lambda)^+(u_\lambda-u)^+-\alpha v_\lambda^2\big ((u_\lambda-u)^+\big )^2\Big]\varphi_R^2 \,.
\end{equation}
A similar computation for $I_4$ yields
\begin{equation}\label{I4new2}
I_4\leq\int_{\Sigma_\lambda} \Big[(\alpha u_\lambda v_\lambda-\omega)(v-v_\lambda)^+(u_\lambda-u)^+ +\alpha u v_\lambda (v-v_\lambda)^+(u_\lambda-u)^+-\alpha u^2\big ((v-v_\lambda)^+\big )^2\Big]\varphi_R^2\,.
\end{equation}
We sum the two estimates \eqref{I2new2} and \eqref{I4new2}, use that
$$2uv_\lambda(v-v_\lambda)^+(u_\lambda-u)^+-v_\lambda^2\big ((u_\lambda-u)^+\big )^2-u^2\big ((v-v_\lambda)^+\big )^2\leq 0$$ to find
\begin{align}
I_2+I_4
	&\leq \int_{\Sigma_\lambda}\Big[(v-v_\lambda)^+(u_\lambda-u)^+
\big ( (\alpha u_\lambda v_\lambda-\omega)+ (\alpha u v -\omega)\big )\Big]{\varphi_R^2}.
\end{align}
Because of \eqref{mainBC}, for $\lambda$ large enough, in $\Sigma_\lambda$, $u$ tends to $b$, $v$ tends to $a$ and $uv$ tends to $ab=\omega/\alpha$. Moreover in the support of $(v-v_\lambda)^+$, $v\geq v_\lambda\geq a$, so $v_\lambda$ also tends to $a$, and in the support of $(u_\lambda-u)^+$, $b\geq u_\lambda\geq u$, so $u_\lambda$ tends to $b$. This implies that for $\lambda$ large enough, in $\Sigma_\lambda$,  $u_\lambda v_\lambda$ also tends to $ab=\omega/\alpha$.
 Therefore, for {some small} $\varepsilon > 0$ {which will be fixed later}, there exists $\bar \lambda$ large enough, so that for $\lambda\geq \bar \lambda$, in $\Sigma_\lambda\cap S$, where $S$ is the intersection of the supports of  $(u_\lambda-u)^+$ and $(v-v_\lambda)^+$, $$|(\alpha u_\lambda v_\lambda-\omega)+ (\alpha u v -\omega)|\leq \varepsilon $$ and
 \begin{equation}\label{I24}
 I_2+I_4\leq \frac \varepsilon 2\int_{\Sigma_\lambda}\Big[\big ( (v-v_\lambda)^+\big )^2+\big ( (u_\lambda-u)^+
\big )^2\Big] {\varphi_R^2}.
\end{equation}

\noindent{\em Estimate of $I_1$ and $I_3$:}  By the mean value theorem, there exist $\xi_1(x)\in (u(x),u_\lambda(x))$ and $\xi_2(x) \in (v_\lambda(x),v(x))$ such that
 \begin{equation}
 g(u_\lambda,v_\lambda)-g(u,v)=\frac {\partial g}{\partial u}(\xi_1,v_\lambda)(u_\lambda-u)+\frac {\partial g}{\partial v}(u,\xi_2)(v_\lambda-v).
 \end{equation}
 Note that, for $\lambda$ large, in $\Sigma_\lambda \cap S$, $u$ and $u_\lambda$ tend to $b$ and $v$ and $v_\lambda$ tend to $a$. Hence, as $g(u,v)=u(1-u^2-v^2)$,
  we find that
  $$\frac {\partial g}{\partial u}(\xi_1,v_\lambda) = 1-3\xi_1^2-v_\lambda^2\to -2b^2 \quad \hbox{ and } \quad \frac {\partial g}{\partial v}(u,\xi_2) = 2u\xi_2 \to 2ab.
  $$
Therefore, { in view of \eqref{I1} and by enlarging $\bar \lambda$ if necessary} we have for $\lambda \geq \bar\lambda$ that
 \begin{equation}\label{I1new2}
 I_1\leq  \int_{\Sigma_\lambda}
 \Big[(-2b^2+\varepsilon) \big ((u_\lambda-u)^+\big )^2+ (2ab+\varepsilon) (v-v_\lambda)^+(u_\lambda-u)^+ \Big]\varphi_R^2 \,.
 \end{equation}
 We argue similarly for $I_3$ and find
  \begin{align}
 I_1+I_3
 	&\leq  \int_{\Sigma_\lambda}(-(a-b)^2+2\varepsilon)\left [ \big ((u_\lambda-u)^+\big )^2+ \big (v-v_\lambda)^+\big)^2\right ] \varphi_R^2.
	\label{I13}
 \end{align}

Recall that $(a-b)^2$ is positive as soon as \eqref{cond} holds. {In the sequel, we assume that $(a-b)^2 > 3\varepsilon$. Inserting \eqref{I24} and \eqref{I13} into \eqref{Eq:LR2R1234}, we have for $\lambda>\bar \lambda$ and large $R$ that
\begin{equation}
{\mathcal L}_\lambda (R)
	\leq \theta{\mathcal L}_\lambda (2R)
		 + \Big(-(a-b)^2+ 3\varepsilon  \Big) \mathcal{J}_\lambda(R)
	\leq  \theta{\mathcal L}_\lambda (2R).
	\label{Eq:27VI19-B1}
\end{equation}
We are now able to apply Lemma \ref{lemLR}.  We have
 that $\mathcal L_\lambda(R)\leq C R^N$ since $|\nabla u|, |\nabla v| \in L^\infty(\mathbb{R}^N)$ by elliptic estimates and the $L^\infty$ bound of Lemma \ref{Lem:14V19-LB}.
 We fix $\theta:= 2^{-(N+1)}$ and some $\varepsilon > 0$ such that $(a - b)^2 - 3\varepsilon > 0$. Then \eqref{Eq:27VI19-B1} is in effect for $\lambda>\bar \lambda$ and large $R$} and Lemma \ref{lemLR} then yields that
\[
\mathcal L_\lambda (R)=0 \text{ for all $\lambda > \bar\lambda$ and large $R$}.
\]
Recalling that $u=u_\lambda$ and $v=v_\lambda$ on $\pa \Sigma_\lambda$, we reach the conclusion.
\end{proof}
We now have to prove that $\tilde \lambda:= \inf \Lambda$ (with $\Lambda$ defined in \eqref{moving thesis 1}) is  $-\infty$ to complete the proof of Proposition \ref{propmonotxN}:
\begin{lem}
We have $\tilde \lambda =-\infty$.
\end{lem}
\begin{proof}
 Assume by contradiction that $\tilde \lambda$ is finite. Then, $\Lambda = [\tilde \lambda,+\infty)$, and  there exist sequences $(\lambda_i)$ with $ \lambda_i \in (-\infty,\tilde \lambda)$ and $(x^i)$ with $x^i \in \Sigma_{\la_i}$ such that $\la_i \to \tilde \lambda$ as $i \to \infty$, and at least one of the two holds:
\begin{subequations}
\begin{align}
& u_{\la_i}(x^i) > u(x^i)  \text{\, \, for every $i$, }\label{eq8} \hbox{ or }\\
& v_{\la_i}(x^i)<v(x^i) \text{ \,\, for every $i$}. \label{eq13}
\end{align}
\end{subequations}

Assume that \eqref{eq8} holds; the other case can be treated similarly. We claim that the sequence $(x_N^i) \subset \R$ is bounded. If not, as $x_N^i>\la_i$ and $\la_i$ is bounded, up to a subsequence $x_N^i \to +\infty$ as $i \to \infty$. It follows that $2 \la_i -x_N^i \to -\infty$, and in light of assumption \eqref{mainBC} we obtain
\[
\lim_{i \to \infty} u_{\la_i}(x^i) = \lim_{i \to \infty} u((x^i)',2\la_i -x_N^i)  = a \quad \text{and} \quad \lim_{i \to \infty} u(x^i) = b,
\]
in contradiction with \eqref{eq8} for $i$ sufficiently large. Hence the claim is proved and, up to a subsequence, $x_N^i \to \bar x_N$ as $i \to \infty$.

Let us set
\[
u^i(x):= u((x^i)'+x',x_N) \quad \text{and} \quad v^i(x):= v((x^i)'+x',x_N).
\]
Since $(u,v)$ is bounded (in view of \eqref{keybound}), by standard elliptic estimates $|\nabla^k u|, |\nabla^k v| \in L^\infty(\R^N)$, for $k = 1, 2, \ldots$
Thus, after extracting a subsequence if necessary, $(u^i,v^i)$ converges in $\mathcal{C}^2_{loc}(\R^N)$ to a limit $(\bar u, \bar v)$, still solution of \eqref{mainsystem}.

We wish to show that $\bar x_N=\tilde \lambda$. From  equation \eqref{eq8}, we obtain
\begin{equation}\label{eq9}
\begin{split}
\bar  u_{\tilde \lambda}(0',\bar x_N) &= \bar u(0',2{\tilde \lambda} - \bar x_N) = \lim_{i \to \infty} u( (x^i)', 2\la_i-x_N^i) \\
&= \lim_{i \to \infty} u_{\la_i}(x^i) \ge \lim_{i \to \infty} u(x^i) = \bar u (0',\bar x_N).
\end{split}
\end{equation}
On the other hand, we observe that $((x^i)'+x',x_N) \in \Sigma_{\tilde \lambda}$ whenever $(x',x_N) \in \Sigma_{\tilde \lambda}$, and by definition $u_{\tilde \lambda} \le u$ in $\Sigma_{\tilde \lambda}$. Consequently, by the convergence of $u^i$ to $\bar u$ we deduce that
\begin{align*}
\bar u_{\tilde \lambda}(x',x_N) &= \lim_{i \to \infty} u^i(x',2{\tilde \lambda} - x_N) = \lim_{i \to \infty} u((x^i)'+x',2 {\tilde \lambda}- x_N) \\
& \le \lim_{i \to \infty} u((x^i)'+x',x_N) = \lim_{i \to \infty} u^i(x',x_N) = \bar u(x',x_N)
\end{align*}
for every $(x',x_N) \in \Sigma_{\tilde \lambda}$. Analogously, as $v_{\tilde \lambda} \ge v$ in $\Sigma_{\tilde \lambda}$, we have $\bar v_{\tilde \lambda} \ge \bar v$ in $\Sigma_{\tilde \lambda}$.\\
Now
\begin{equation}
\left\{\begin{array}{ll}
-\Delta (\bar u - \bar u_{\tilde \lambda}) + c(x) (\bar u- \bar u_{\tilde \lambda})
	&=  (\bar v_{\tilde \lambda} - \bar v) (\alpha \bar u\bar v -\omega+\alpha \bar u_{\tilde \lambda} \bar v_{\tilde \lambda}+\bar u_{\tilde \lambda}(\bar v_{\tilde \lambda}+\bar v))\\
\bar u - \bar u_{\tilde \lambda} \ge 0 & \text{in $\Sigma_{\tilde \lambda}$} \\
\bar u - \bar u_{\tilde \lambda} = 0 & \text{on $\pa \Sigma_{\tilde \lambda}$},
\end{array}\right.
\label{eq12}
\end{equation}
with $c \in C^0(\Sigma_{\tilde \lambda})$ defined by
$$
c(x)= \alpha \bar v(x)\bar v_{\tilde \lambda}(x)+ \bar v^2(x) + \bar u^2(x) + \bar u \bar u_{\tilde \lambda}(x)  +  \bar u_{\tilde \lambda}^2(x) -1 .$$
Because of \eqref{keybound}, the right hand side on the first line of \eqref{eq12} is nonnegative,
hence, the strong Maximum Principle  implies that necessarily $\bar u- \bar u_{\tilde \lambda} > 0$ in $\Sigma_{\tilde \lambda}$, and a comparison with \eqref{eq9} reveals that $\bar x_N={\tilde \lambda}$, as desired.

At this point we are ready to reach a contradiction. On the one hand, by the absurd assumption \eqref{eq8}
\[
0< u_{\la_i}(x^i) - u(x^i) = u^i(0', 2\la_i -x^i_N)-u^i(0',x_N) = 2\pa_N u^i (0', \xi^i)(\la_i-x^i_N) \qquad \forall i,
\]
for some $\xi^i \in (2\lambda_i -x_N^i,x_N^i)$. As $\la_i < x_N^i$ for every $i$ this implies $\pa_N u^i(x', \xi^i_N)< 0$ for every $i$, and passing to the limit we infer that
\begin{equation}\label{eq11}
\pa_N \bar u(0', {\tilde \lambda}) \le 0,
\end{equation}
where we used the fact that $\la_i \le \xi^i \le x_N^i$ with $\la_i,x_N^i \to {\tilde \lambda}$.

On the other hand, thanks to  \eqref{eq12} and the fact that $\bar u-\bar u_{\tilde \lambda}>0$ in $\Sigma_{\tilde \lambda}$, the Hopf Lemma implies that
\[
-2\pa_N \bar u(0',{\tilde \lambda}) =\pa_{-e_N} (\bar u(0',{\tilde \lambda}) -\bar u_{\tilde \lambda}(0',{\tilde \lambda}))<0,
\]
in contradiction with \eqref{eq11}.

The above argument establishes that \eqref{eq8} cannot occur. With minor changes, we can show that also \eqref{eq13} cannot be verified, and in conclusion $\tilde \lambda$ cannot be finite.
\end{proof}

\begin{proof}[Proof of Proposition \ref{propmonotxN}]
By \eqref{moving thesis 1}, we directly deduce that $\pa_N u \ge 0$ and $\pa_N v \le 0$ in $\R^N$. Since
\[
\begin{cases}
-\Delta (\pa_N u) + (3 u^2  + (\alpha +1) v^2 -1) \pa_N u = (\omega -2 (\alpha +1) uv) \pa_N v \ge 0 & \text{in $\R^N$} \\
-\Delta (\pa_N v) + (3 v^2  + (\alpha +1) u^2 -1) \pa_N v = (\omega -2 (\alpha +1) uv) \pa_N u \le 0 & \text{in $\R^N$},
\end{cases}
\]
the strict inequality follows by the strong Maximum Principle.
\end{proof}

\subsection{One-dimensional symmetry}\label{sec: sym}

We extend the monotonicity in $x_N$ to  all the directions of the open upper hemisphere $\mathbb{S}^{N-1}_+:=\left\{ \nu \in \mathbb{S}^{N-1}: \langle \nu, e_N \rangle >0 \right\}$. We follow the structure of proof in \cite{FSS}, introduced in \cite{Fa99} and in \cite{FaSo}, though the specificity of our system requires new estimates.

\begin{prop}\label{prop: symmetry}
For every $\nu \in \mathbb{S}^{N-1}_+$, we have
\[
\pa_\nu u >0 \quad \text{and} \quad \pa_\nu v<0 \qquad \text{in $\R^N$}.
\]
In particular, $u$ and $v$ depend only on $x_N$.
\end{prop}

We divide the proof into several steps.

\begin{lem}\label{lem: step 1 2}
Let $\sigma >0$ be arbitrarily chosen. There exists an open neighborhood $\mathcal{O}_{e_N}$ of $e_N$ in $\mathbb{S}^{N-1}$ such that
\[
\frac{\pa u}{\pa \nu}  (x) > 0 \quad \text{and} \quad \frac{\pa v}{\pa \nu}(x) <0 \quad \forall x \in \overline{S_\s}, \ \forall \nu \in \mathcal{O}_{e_N},
\]
where $S_\sigma:= \R^{N-1} \times (-\sigma,\sigma)$.
\end{lem}

\begin{proof}
Let $\sigma>0$ be arbitrarily chosen. Firstly, we claim that there exists $\eps=\eps(\sigma)>0$ such that
\begin{equation}\label{th step 1}
\pa_N u(x) \ge \eps \quad \text{and} \quad \pa_N v(x) \le -\eps \quad \forall x \in \overline{S}_\sigma.
\end{equation}
By contradiction, assume that there exists a sequence  $(x^i)$, with   $ x^i \in S_\sigma$, such that at least one of the two following equalities  holds :
\begin{subequations}
\begin{align}
& \lim_{i \to +\infty} \pa_N u(x^i)= 0 \label{eq14} \\
\hbox{or }& \lim_{i \to +\infty} \pa_N v(x^i)= 0 \label{eq15}.
\end{align}
\end{subequations}
We define
\[
u^i(x):= u(x+x^i) \quad \text{and} \quad v^i(x):= v(x+x^i).
\]
The sequence $\{(u^i,v^i)\}$ is uniformly bounded in $W^{1,\infty}(\R^N)$, and hence by elliptic regularity $(u^i,v^i) \to (\bar u,\bar v)$ in $\mathcal{C}^2_{\rm loc}(\R^N)$ up to a subsequence, where $(\bar u, \bar v)$ is still a solution to \eqref{mainsystem}-\eqref{mainBC}, which also satisfies $\pa_N \bar u \geq 0$ and $\pa_N \bar v \leq 0$ on $ \RR^N$. The strong maximum principle and the condition at infinity \eqref{mainBC} then imply that $\pa_N \bar u > 0$ and $\pa_N \bar v < 0$ on $ \RR^N$, and this contradicts \eqref{eq14} or
\eqref{eq15}. This completes the proof of claim \eqref{th step 1}.

Now we claim that
\begin{equation}\label{th step 2}
\text{The map $\nu \mapsto (\pa_\nu u, \pa_\nu v)$ is in $\mathcal{C}^{0,1} \left(\mathbb{S}^{N-1}, \left(\mathcal{C}^0(\R^N)\right)^2\right)$}.
\end{equation}
This is a simple consequence of the globlal Lipschitz continuity of $(u,v)$, which implies that
\[
\left|\frac{\pa u}{\pa \nu_1}(x) -\frac{\pa u}{\pa \nu_2}(x) \right| + \left|\frac{\pa v}{\pa \nu_1}(x) -\frac{\pa v}{\pa \nu_2}(x) \right| \le 2 C |\nu_1-\nu_2|
\]
for every $x \in \R^N$.

Combining \eqref{th step 1} and \eqref{th step 2}, the conclusion follows.
\end{proof}

\begin{lem}\label{lem: step 3} The function
$u$ is strictly increasing and $v$ is strictly decreasing with respect to all  unit vectors in an open neighborhood of $e_N$ in $\mathbb{S}^{N-1}$.
\end{lem}

\begin{proof}
Firstly, we write down the equations satisfied by the directional derivatives $ u_\nu= \pa_\nu u$ and $v_\nu=\pa_\nu v$:
\begin{equation}\label{linearized}
\begin{cases}
-\Delta u_\nu +  u_\nu (3u^2+(\alpha+1)v^2-1)+ v_\nu (2(\alpha+1)uv-\omega)=0 \\
-\Delta v_\nu +  v_\nu (3v^2+(\alpha+1)u^2-1)+ u_\nu (2(\alpha+1)uv-\omega)=0
\end{cases} \text{in $\R^N$}.
\end{equation}

Fix some $\sigma > 0$ for the moment and let $\mathcal{O}_{e_N}$ be the neighborhood of $e_N$ given by Lemma \ref{lem: step 1 2}. {We will show that $u_\nu \ge 0$ and $v_\nu \le 0$ for all $\nu \in \mathcal{O}_{e_N}$ by applying Lemma \ref{lemLR} to the quantity
\[
\mathcal{I}_R := \int_{\mathcal{C}_R} \big[|\nabla u_\nu^-|^2 + |\nabla v_\nu^+|^2\big],
\]
where $\mathcal{C}_{R}:= \Sigma_\sigma \cap B_{R}$ and $\Sigma_\sigma :=\{x_N > \sigma\}$. The conclusion then follows from the strong maximum principle.
}

We test the first equation in \eqref{linearized} with $u_\nu^- \varphi_R^2$ where $\varphi_R$ is chosen exactly as in Lemma \ref{monotinf}. Using the bounds \eqref{keybound} and the fact that $u_\nu \ge 0$ on $\{x_N = \sigma\}$ (due to Lemma \ref{lem: step 1 2}), we obtain
\begin{align*}
\int_{\mathcal{C}_R} |\nabla u_\nu^-|^2 & \le -2 \int_{\mathcal{C}_{2R}} u_\nu^- \varphi_R \nabla u_\nu^- \cdot \nabla \varphi_R  \\
&\quad - \int_{\mathcal{C}_{2R}} (3u^2+(\alpha+1) v^2-1) (u_\nu^- \varphi_R)^2  + \int_{\mathcal{C}_{2R}} \varphi_R^2 (2(\alpha+1) uv -\omega) u_\nu^- v_\nu^+ \\
& \le \theta \int_{\mathcal{C}_{2R}} |\nabla u_\nu^-|^2 + \int_{\mathcal{C}_{2R}} (u_\nu^- \varphi_R)^2\left( \frac{4}{\theta R^2} + \sup_{\Sigma_\s} (-3u^2-(\alpha+1) v^2+1) \right)  \\
&\quad + \int_{\mathcal{C}_{2R}} \varphi_R^2 (2(\alpha+1) uv -\omega) u_\nu^- v_\nu^+ ,
\end{align*}
where $0<\theta<2^{-N}$. In a similar way, we find for $v_\nu^+$
\begin{align*}
\int_{\mathcal{C}_R} |\nabla v_\nu^+|^2 & \le \theta \int_{\mathcal{C}_{2R}} |\nabla v_\nu^+|^2 +  \int_{\mathcal{C}_{2R}} (v_\nu^+ \varphi_R)^2 \left( \frac{4}{\theta R^2} + \sup_{\Sigma_\s} (-3v^2-(\alpha+1)u^2+1) \right) \\
&\quad +  \int_{\mathcal{C}_{2R}}\varphi_R^2 (2(\alpha+1) uv -\omega) u_\nu^- v_\nu^+ .
\end{align*}
We notice that if $\s>0$ is sufficiently large, since $u$ tends to $b$ and $v$ tends to $a$ for $x_N$ large, then, in $\Sigma_\s$,
$$2(\alpha+1) uv -\omega \to 2(\alpha+1) ab -\omega,\ 3u^2+(\alpha+1) v^2-1\to 2b^2+\alpha a^2,$$
$$\hbox{ and }
3v^2+(\alpha+1)u^2-1\to 2a^2+\alpha b^2.$$
{Thus, for any small $\delta > 0$, we can choose $\sigma$ and $R$ sufficiently large so that
\begin{multline}
\mathcal{I}_R
	\leq \theta \mathcal{I}_{2R} - \int_{\mathcal{C}_{2R}} \Big[(2b^2+\alpha a^2 - \delta) (u_\nu^-)^2 + (2a^2+\alpha b^2 - \delta)(v_\nu^+)^2 \\
		 - 2(2(\alpha+1) ab -\omega + \delta)|u_\nu^-||v_\nu^+| \Big]\varphi_R^2.
	\label{Eq:16VII19-E1}
\end{multline}
We point out that
$$(2b^2+\alpha a^2 - \delta) (u_\nu^-)^2 + (2a^2+\alpha b^2 - \delta)(v_\nu^+)^2\geq 2
 \sqrt {(2b^2+\alpha a^2 - \delta)(2a^2+\alpha b^2 - \delta)}|u_\nu^-||v_\nu^+|.$$
and that, by \eqref{eqab} and \eqref{cond},
\begin{multline*}
(2b^2+\alpha a^2)(2a^2+\alpha b^2)  - (2(\alpha+1) ab -\omega)^2\\
	= 2\alpha (a^2 + b^2)^2  - 3\alpha(\alpha + 4) a^2b^2  + 4\omega(\alpha + 1) ab - \omega^2
	= 2\alpha\Big(1 - \frac{4\omega^2}{\alpha^2}\Big) > 0.
\end{multline*}
Hence, by choosing first small $\delta$ and then large $\sigma$ from the start, we have for all sufficiently large $R$ that the integral on the right hand side of \eqref{Eq:16VII19-E1} is non-negative.
}
As a consequence, we infer that
\begin{equation}
\mathcal{I}_R
 \le \theta \mathcal{I}_{2R} \text{ for all $R$ sufficiently large}.
  \end{equation}
We can now apply  Lemma \ref{lemLR} to find that $\mathcal{I}_R = 0$ for all large $R$. It follows that $u_\nu \ge 0$ and $v_\nu \le 0$ in $\Sigma_\s=\{x_N > \sigma\}$. Arguing exactly in the same way, we can show that the same conditions are satisfied in $\{x_N <-\sigma\}$. By Lemma \ref{lem: step 1 2}, we deduce that $u_\nu \ge 0$ and $v_\nu \le 0$ in $\R^N$ for every $\nu \in \mathcal{O}_{e_N}$, with both $u_\nu \not \equiv 0$ and $v_\nu \not \equiv 0$. {In view of \eqref{keybound} and \eqref{linearized}, the conclusion follows from the strong maximum principle.}
\end{proof}

\begin{proof}[Proof of Proposition \ref{prop: symmetry}]
Here we can essentially apply the same argument used in step 4 of Proposition 6.1 in \cite{FaSo}. We report the details for completeness. Let $\Omega$ be the set of the directions $\nu \in \S^{N-1}_+$ for which there exists an open neighborhood $\mathcal{O}_\nu \subset \S^{N-1}_+$ of $\nu$ such that
\[
\pa_\mu u>0 \quad \text{and} \quad \pa_\mu v<0 \quad \text{in $\R^N$}, \ \forall \mu \in \mathcal{O}_\nu.
\]
The set $\Omega$ is open by definition, and contains $e_N$ by Lemma \ref{lem: step 3}. Since $\S^{N-1}_+$ is arc-connected, if we can show that $\pa \Omega \cap \S^{N-1}_+=\emptyset$, then we conclude that $\Omega=\S^{N-1}_+$, as desired. Thus, let us suppose by contradiction that $\bar \nu \in \pa \Omega \cap \S^{N-1}_+$ (notice in particular that $\langle e_N, \bar \nu\rangle >0$). By definition, there exists $(\nu_n) \subset \Omega$ such that $\nu_n \to \bar \nu$. As
\[
\pa_{\nu_n} u >0 \quad \text{and} \quad \pa_{\nu_n} v<0 \quad \text{in $\R^N$}, \ \forall n,
\]
by continuity
 \[
\pa_{\bar \nu} u  \ge 0 \quad \text{and} \quad \pa_{\bar \nu} v\le 0 \quad \text{in $\R^N$}.
\]
By the strong maximum principle, recalling that $(u_{\bar \nu}, v_{\bar \nu})$ solves \eqref{linearized}, either $u_{\bar \nu} \equiv 0$ or $u_{\bar \nu}>0$ in $\R^N$, and analogously either $v_{\bar \nu} \equiv 0$ or $v_{\bar \nu}<0$ in $\R^N$. The alternatives $u_{\bar \nu}  \equiv 0$ and $v_{\bar \nu} \equiv 0$ are in contradiction with assumption \eqref{mainBC}, since $\bar \nu$ is not orthogonal to $e_N$, and hence
\begin{equation}\label{2431}
\pa_{\bar \nu} u>0\quad \text{and} \quad  \pa_{\bar \nu} v<0 \qquad \text{in $\R^N$}.
\end{equation}
Having established \eqref{2431}, it is possible to adapt the same proof of Lemmas \ref{lem: step 1 2} and \ref{lem: step 3}, with $\bar \nu$ instead of $e_N$, to deduce that $u_\nu >0$ and $v_\nu<0$ in $\R^N$ in all the directions of an open neighborhood $\mathcal{O}_{\bar \nu}$ of $\bar \nu$ in $\mathbb{S}^{N-1}_+$. Thus, we have that $\bar \nu \in \Omega \cap \pa \Omega$, in contradiction with the openness of $\Omega$.  This shows that $\pa \Omega \cap \S^{N-1}_+=\emptyset$ which, as already observed, implies $\Omega = \S^{N-1}_+$.

Finally, the fact that $\Omega = \S^{N-1}_+$ implies that both $\partial_\tau u\equiv 0$ and $\partial_\tau v\equiv 0$ for every $\tau \in \mathbb{S}^{N-1}$ orthogonal to $e_N$, which proves the last assertion.
\end{proof}

%

%----------------------------------------------------------------------------%

\section{Existence and uniqueness of positive $1D$ solutions when $2\omega < \alpha$. Proof of Theorem \ref{maintheo}.}

In this section, we assume that $0 < \omega < \frac{1}{2}\alpha$ unless otherwise stated. By Proposition \ref{prop: symmetry}, positive solutions to \eqref{mainsystem}-\eqref{mainBC} depend only on $x_N$ and are monotone. To conclude the proof of Theorem \ref{maintheo}, it remains to prove the uniqueness up to translations of such one-dimensional solutions.

We are led to consider on $\RR$ the system
\begin{equation}
\left\{\begin{array}{ccc}\vone''
	&=& \RHSA(\vone, \vtwo),\\
\vtwo''
	&=& \RHSB(\vone, \vtwo),
	\end{array}\right.
	\label{Eq:Osys}
\end{equation}
subject to
\begin{equation}
\text{$(\vone, \vtwo) \rightarrow (a,b)$ at $-\infty$ and $(\vone, \vtwo) \rightarrow (b,a)$ at $\infty$.}
	\label{Eq:26VI19-OBC}
\end{equation}

The main result of this section is:
\begin{proposition}\label{Thm:EUMODE}
Suppose that $0 < \omega < \frac{1}{2}\alpha$. Then there exist positive solutions to \eqref{Eq:Osys}-\eqref{Eq:26VI19-OBC}, and these solutions are translations of one another, i.e. if $(\vone, \vtwo)$ and $(\bar \vone, \bar \vtwo)$ both satisfy \eqref{Eq:Osys}-\eqref{Eq:26VI19-OBC}, then there is a constant $T$ such that
\[
\bar \vone(x) = \vone(x + T) \text{ and } \bar \vtwo(x) = \vtwo(x + T).
\]
Furthermore, $\vone' > 0$ and $\vtwo' < 0$ in $\RR$.
\end{proposition}

\begin{proof}[Proof of Theorem \ref{maintheo}]
The result is a consequence of Propositions \ref{prop: symmetry} and \ref{Thm:EUMODE}.
\end{proof}

The proof of the `uniqueness' part in Proposition \ref{Thm:EUMODE} uses the sliding method (cf. \cite{BCN97,BN90,BNsliding}) with the help of the bounds \eqref{keybound} as well as the following lemma on the asymptotic behavior of solutions.

Let
\begin{align}
\lambda_\pm
	&:= \sqrt{\frac{1}{2} \Big(  (\alpha + 2) \pm \sqrt{(\alpha - 2)^2 + \frac{32\omega^2}{\alpha}}\Big)},
	\label{Eq:04II19-A1}\\
\mu
	&:= \frac{2\omega (\alpha + 2)}{\alpha\Big( (\alpha - 2) \sqrt{1 - \frac{4\omega^2}{\alpha^2} } + \sqrt{(\alpha-2)^2  + \frac{32\omega^2}{\alpha} } \Big)} > 0,
	\label{Eq:mudef}
\end{align}
which are related to the eigenvalues and eigenvectors of the linearized operator associated with \eqref{Eq:Osys} near the critical point $(a,b)$.
 We refer to Appendix \ref{App:Asym} for a brief discussion on the origin of these constants.

\begin{lemma}\label{Lem:HypExpDecay.Pos}
Suppose that $0 < \omega < \frac{1}{2}\alpha$. Let $(\vone, \vtwo)$ be a positive solution of \eqref{Eq:Osys}-\eqref{Eq:26VI19-OBC} and $\lambda_-$ and $\mu$ be defined by \eqref{Eq:04II19-A1}-\eqref{Eq:mudef}. Then the limits
\begin{equation}
\ell_1 := \lim_{x \rightarrow \infty} (b - \vone(x))e^{\lambda_- x}\text{ and } \ell_2 := \lim_{x \rightarrow \infty} (\vtwo(x) - a)e^{\lambda_- x} \text{ exist,}
	\label{Eq:06II19-elEx}
\end{equation}
and satisfy
\begin{equation}
\ell_2 = \mu \ell_1 > 0.
	\label{Eq:06II19-elRel}
\end{equation}
\end{lemma}

It should be noted that the asymptotic behavior of solutions changes somewhat when $\omega = 0$. See Lemma \ref{Lem:HypExpDecay.Zero} in the appendix.

An easy variational argument gives existence of the positive solutions to \eqref{Eq:Osys} on finite intervals. The sliding method can be adapted to this case yielding:

\begin{lemma}\label{Lem:EUMODE-FR}
Suppose that $0 < \omega < \frac{1}{2}\alpha$. For every $R \in (0,\infty)$, there exists a unique positive solution to \eqref{Eq:Osys} in $(-R,R)$ satisfying
\begin{equation}
\text{$(\vone(-R), \vtwo(-R)) = (a,b)$  and $(\vone(R), \vtwo(R)) = (b,a)$.}
	\label{Eq:OBC}
\end{equation}
Furthermore, $\vone' > 0$ and $\vtwo' < 0$ in $(-R,R)$.
\end{lemma}

\subsection{Proof of Proposition \ref{Thm:EUMODE}}

Let us assume Lemmas \ref{Lem:HypExpDecay.Pos} and \ref{Lem:EUMODE-FR} for the moment and proceed with the proof of Proposition \ref{Thm:EUMODE}. Lemma \ref{Lem:EUMODE-FR} will be proved in the next subsection. Lemma \ref{Lem:HypExpDecay.Pos} follows from a routine asymptotic analysis near a hyperbolic critical point for ODEs. Its proof is postponed to Appendix \ref{App:Asym}.

\begin{proof}
1. We prove the existence of a solution to \eqref{Eq:Osys}-\eqref{Eq:26VI19-OBC} by sending $R \rightarrow \infty$ in Lemma \ref{Lem:EUMODE-FR}, where some care is needed to show that the solutions on finite intervals do not flatten to the constant solutions $(a,a)$ or $(b,b)$.

For $n =1, 2, \ldots $, let $(\vone_n, \vtwo_n)$ be the positive solution to \eqref{Eq:Osys} obtained in Lemma \ref{Lem:EUMODE-FR} with $R = n$. Fix $x_n \in (-n,n)$ such that $u_n(x_n) = \frac{1}{2}(a + b)$. Define $l_n = -n - x_n$, $r_n = n - x_n$, and
\[
(\tilde \vone_n(x), \tilde \vtwo_n(x)) = (\vone_n(x + x_n),\vtwo_n(x + x_n)) \text{ for } x \in [l_n,r_n].
\]

By Lemma \ref{Lem:EUMODE-FR}, $a \leq u_n, v_n \leq b$. Using elliptic estimates on unit closed subintervals of $[l_n,r_n]$, we have

\begin{equation}
|\tilde u_n^{(k)}| + |\tilde v_n^{(k)}|  \leq C \quad {\text {in}} \quad [l_n,r_n] \quad {\text {for}} \quad k=0,1,2,3,
\label{Eq:k-estimate}
\end{equation}
where $C$ is a positive constant independent of $n$ and $k$. Then, passing to a subsequence if necessary, we may assume that $l_n \rightarrow l_* \in [-\infty,0]$, $r_n \rightarrow r_* \in [0,\infty]$ (where $l_*$ and $r_*$ cannot be simultaneously finite), $(\tilde \vone_n, \tilde \vtwo_n)$ converges in $\mathcal{C}^2_{\rm loc}(l_*,r_*)$ to some $(\vone_*,\vtwo_*)$ satisfying \eqref{Eq:Osys}, $\vone_*' \geq 0$, $\vtwo_*' \leq 0$  in $(l_*, r_*)$.

Note that for each $n$, the Hamiltonian
\[
h_n := \frac{1}{2} (|\tilde\vone_n'|^2 + |\tilde\vtwo_n'|^2) - \frac{1}{4} (1 - \tilde\vone_n^2 - \tilde\vtwo_n^2 )^2 - \frac{\alpha}{2} (\tilde\vone_n\tilde\vtwo_n - \frac{\omega}{\alpha})^2
\]
is constant in $[l_n, r_n]$. As $(\tilde \vone(R), \tilde\vtwo(R)) = (b,a)$, it follows that $h_n \geq 0$ and so
\[
h_* := \frac{1}{2} (|\vone_*'|^2 + |\vtwo_*'|^2) - \frac{1}{4} (1 - \vone_*^2 - \vtwo_*^2 )^2 - \frac{\alpha}{2} (\vone_* \vtwo_* - \frac{\omega}{\alpha})^2 \geq 0.
\]

As said above, one has that $l_* = -\infty$ or $r_* = \infty$ (or both). We will only treat the case that $r_* = \infty$; the other case can be dealt with similarly. In this case we have $u_*(x) \geq \frac{1}{2}(a + b)$ for every $ x > 0$, since $ x > 0 \geq l_*$ implies $ \tilde u_n(x) \geq \tilde u_n(0) = \frac{a+b}{2} $ for large $n$. Then, by the monotonicity of $\vone_*$ and $\vtwo_*$, as $x \rightarrow \infty$, $(\vone_*(x),\vtwo_*(x))$ has a limit, say $(\tilde b, \tilde a)$, which satisfies $\frac{1}{2}(a + b) \leq \tilde b \leq b$ and $a \leq \tilde a \leq b$. By \eqref{Eq:Osys}, $(\vone_*''(x),\vtwo_*''(x))$ tends to $(\RHSA(\tilde b,\tilde a), \RHSB(\tilde b, \tilde a))$ as $x \rightarrow \infty$. Applying the mean value theorem to $\vone|_{[n, n+1/2]}$, we can find $\xi_n \in (n, n+1/2)$ such that $\vone_*'(\xi_n) \rightarrow 0$. Likewise, there exist $\eta_n \in (\xi_n, \xi_{n+1})$ such that $\vone_*''(\eta_n) \rightarrow 0$. It follows that $\RHSA(\tilde b,\tilde a) = 0$. This then implies that $\sup_{[\xi_n-2,\xi_n + 2]} u_*'  \rightarrow 0$, and so $u_*' (x) \rightarrow 0$ as $x \rightarrow \infty$. Similarly, $\RHSB(\tilde b,\tilde a) = 0$ and $v_*' (x) \rightarrow 0$ as $x \rightarrow \infty$.

Now, note that the equation $\RHSA(x,y) = \RHSB(x,y) = 0$ has three solutions in the positive quadrant, namely $(a,b)$, $(b,a)$ and $(c,c)$ where $c = \sqrt{\frac{1 + \omega}{2+\alpha}} \in (a,b)$. Also, as $(u_*'(x),v_*'(x)) \rightarrow 0$ as $x \rightarrow \infty$,
\[
0 \leq h_* = - \frac{1}{4} (1 - \tilde b ^2 - \tilde a^2 )^2 - \frac{\alpha}{2} (\tilde b \tilde a - \frac{\omega}{\alpha})^2
\]
and so $h_* = 0$, $\tilde b ^2 + \tilde a^2 = 1$ and $\tilde b \tilde a = \frac{\omega}{\alpha}$. As $\frac{1}{2}(a + b) \leq \tilde b \leq b$, we thus have $(\tilde b, \tilde a) = (b,a)$, i.e. $(u_*(x),v_*(x)) \rightarrow (b,a)$ as $x \rightarrow \infty$.

Now if $l_*$ is finite, \eqref{Eq:k-estimate} yields
\begin{equation}
|\tilde u_*^{(k)}| + |\tilde v_*^{(k)}|  \leq C \quad {\text {in}} \quad (l_*,r_*) \quad {\text {for}} \quad k=0,1,2,
\label{Eq:k-star-estimate}
\end{equation}
where $C$ is a positive constant independent of $k$, and so $u_*$ extends to a $\mathcal{C}^1$ function in $[l_*, r_*)$.
Now, since for every $x \in (l_*, r_*)$, we have $x \in [l_n, r_n]$ for large $n$, from \eqref{Eq:k-estimate} we also get that
\begin{align*}
|u_*(x) - a|
	&= \lim_{n \rightarrow \infty} |\tilde u_n(x) - a|  =  \lim_{n \rightarrow \infty} |\tilde u_n(x) - \tilde u_n(l_n)| \\
	&\le \limsup_{n \rightarrow \infty} C|x - l_n| = C|x - l_*|
\end{align*}
which leads to $u_*(l_*) =a$. A similar argument gives $v_*(l_*) =b$.
Now, by the strong maximum principle and the Hopf lemma (see the argument in Step 5 of the proof of Proposition \ref{propbounds}), we have $\vone_*'(x) > 0$ and $\vtwo_*'(x) < 0$ for every $x \in [l_*, r_*)$, and using those properties with $x= l_*$ we get $h_* > 0$, which contradicts the previous conclusion that $h_* = 0$. Hence $l_* = -\infty$. As above, this implies that $(u_*(x),v_*(x)) \rightarrow (a,b)$ as $x \rightarrow -\infty$.

We have thus shown that $(\vone_*, \vtwo_*)$  is a positive and strictly monotone solution to \eqref{Eq:Osys}-\eqref{Eq:26VI19-OBC}, as desired.

\medskip
\noindent 2. We use the sliding method to show that positive solutions to \eqref{Eq:Osys}-\eqref{Eq:26VI19-OBC} are translations of one another.

Let $(\vone, \vtwo)$ and $(\bar \vone, \bar \vtwo)$ be two positive solutions to \eqref{Eq:Osys}-\eqref{Eq:26VI19-OBC}. By Lemma \ref{Lem:HypExpDecay.Pos}, the limits
\begin{align*}
\ell_1^+
	&:= \lim_{x \rightarrow \infty} (b - \vone(x)) e^{-\lambda_- x},
	\qquad \bar \ell_1^+
		:= \lim_{x \rightarrow \infty} (b - \bar \vone(x)) e^{-\lambda_- x},\\
\ell_2^+
	&:= \lim_{x \rightarrow \infty} ( \vtwo(x) - a) e^{-\lambda_- x},
	\qquad \bar \ell_2^+
	:= \lim_{x \rightarrow \infty} (\bar \vtwo - a) e^{-\lambda_- x}, \\
\ell_1^-
	&:= \lim_{x \rightarrow -\infty} (\vone(x) - a) e^{\lambda_- x},
	\qquad \bar \ell_1^-
	:= \lim_{x \rightarrow -\infty} (\bar \vone(x) - a) e^{\lambda_- x},\\
\ell_2^-
	&:= \lim_{x \rightarrow -\infty} (b - \vtwo(x) ) e^{\lambda_- x},
	\qquad \bar \ell_2^-
	:= \lim_{x \rightarrow -\infty} (b - \bar \vtwo) e^{\lambda_- x}
\end{align*}
exist and are positive. Thus, in view of \eqref{boundab}, there is some large $T_0$ such that
\[
\vone(x - T_0) \leq \bar \vone(x) \leq \vone(x + T_0) \text{ and } \vtwo(x - T_0) \geq  \bar \vtwo(x) \geq \vtwo(x + T_0) \text{ for all } x \in \RR.
\]
Let
\[
T = \inf\{t \in [-T_0,T_0]: \bar \vone(x) \leq \vone(x + s) \text{ and } \bar \vtwo(x) \geq \vtwo(x + s) \text{ for all } t \leq s \leq T_0, x \in \RR\}.
\]

Set $\tilde \vone(x) = \vone(x + T)$ and $\tilde \vtwo(x) = \vone(x + T)$. Then $\tilde \vone \geq \bar \vone$ and $\tilde \vtwo \leq \bar \vtwo$. The result will follow once we show that $\tilde \vone \equiv \bar \vone$ and $\tilde \vtwo \equiv \bar \vtwo$. Assume by contradiction that this does not hold.

\medskip
\noindent a. We show that
\begin{equation}
\tilde \vone > \bar \vone \text{ and } \tilde \vtwo < \bar \vtwo.
	\label{Eq:09II19-Q1}
\end{equation}

Note that $\tilde \vone \geq \bar \vone$, $\tilde \vtwo \leq \bar \vtwo$, and  $(\tilde \vone, \tilde \vtwo)$ is also a solution to \eqref{Eq:Osys} satisfying \eqref{Eq:OBC}. Also, we have
\begin{eqnarray*}
\RHSA(\tilde \vone, \tilde \vtwo) - \RHSA(\tilde \vone, \bar \vtwo)
	&=& (\tilde \vtwo - \bar \vtwo) \big[(1 + \alpha) \tilde \vone (\tilde \vtwo + \bar \vtwo) - \omega\big]\\
	&\leq& (\tilde \vtwo - \bar \vtwo) \big[(1 + \alpha) \tilde \vone \tilde \vtwo   - \omega\big]\\
	&\stackrel{\eqref{keybound}}{\leq}& (\tilde \vtwo - \bar \vtwo)  \big[(1 + \alpha) \frac{\omega}{\alpha}  - \omega\big]\\
	&\leq& 0,
\end{eqnarray*}
and
\begin{eqnarray*}
\RHSB(\tilde \vone, \tilde \vtwo) - \RHSB(\bar \vone, \tilde \vtwo)
	&=& (\tilde \vone - \bar \vone) \big[(1 + \alpha) \tilde \vtwo (\tilde \vone + \bar \vone) - \omega\big]\\
	&\geq&  (\tilde \vone - \bar \vone)  \big[(1 + \alpha) \tilde \vtwo \tilde \vone - \omega\big]\\
	&\stackrel{\eqref{keybound}}{\geq}&  (\tilde \vone - \bar \vone)  \big[(1 + \alpha) \frac{\omega}{\alpha}  - \omega\big]\\
	&\geq& 0.
\end{eqnarray*}
So we have
\begin{align*}
\tilde \vone''
	&\leq \RHSA(\tilde \vone, \bar \vtwo), \qquad \bar \vone'' = \RHSA (\bar \vone, \bar \vtwo)\\
\tilde \vtwo
	&\geq \RHSB(\bar \vone, \tilde \vtwo), \qquad \bar \vtwo'' = \RHSB(\bar \vone, \bar \vtwo).
\end{align*}
Assertion \eqref{Eq:09II19-Q1} thus follows from the strong maximum principle.

\medskip
\noindent b. We proceed to deduce a contradiction.

 { Define
\begin{align*}
\tilde\ell_1^+
	&:= \lim_{x \rightarrow \infty} (b - \tilde\vone(x)) e^{-\lambda_- x},
\qquad
	\tilde\ell_2^+
	:= \lim_{x \rightarrow \infty} ( \tilde\vtwo(x) - a) e^{-\lambda_- x},\\
\tilde\ell_1^-
	&:= \lim_{x \rightarrow -\infty} (\tilde\vone(x) - a) e^{\lambda_- x},
	\qquad
\tilde\ell_2^-
	:= \lim_{x \rightarrow -\infty} (b - \tilde\vtwo(x) ) e^{\lambda_- x}.
\end{align*}
Recall that, by Lemma \ref{Lem:HypExpDecay.Pos}, $\ell_2^+ = \mu\ell_1^+$, $\ell_1^- = \mu \ell_2^-$ and similar relations hold for the counterparts with bar and tilde on top. By the minimality of $T$ and \eqref{Eq:09II19-Q1}, we have that
\begin{equation}
\bar \ell_2^+ = \tilde \ell_2^+,
	\label{Eq:09II19-Q11+}
\end{equation}
or
\begin{equation}
\bar \ell_2^- = \tilde \ell_2^-,
	\label{Eq:09II19-Q11-}
\end{equation}
or both. Because of the unique correspondence of linearized solutions and nonlinear solutions near a hyperbolic critical point of ODEs \cite[Chapter 13, Theorem 4.5]{CoddingtonLevinson} (with $p = 1$) -- see the argument leading to \eqref{Eq:27VI19-A1} -- we then deduce that $(\bar \vone, \bar \vtwo) \equiv (\tilde \vone,\tilde \vtwo)$, which contradicts \eqref{Eq:09II19-Q1}. We conclude the proof.}
\end{proof}

\subsection{Proof of Lemma \ref{Lem:EUMODE-FR}}

The existence in Lemma \ref{Lem:EUMODE-FR} follows from an easy variational argument. As far as we are concerned with the application of Lemma \ref{Lem:EUMODE-FR} to the proof of Proposition \ref{Thm:EUMODE}, it is enough to show that $u$ and $v$ are monotone. This can be done as in Subsection \ref{mono}. Here we provide an alternative proof which yields also uniqueness, which echoes the argument in Step 2 of the proof of Proposition \ref{Thm:EUMODE}.

We start with an adaptation of Proposition \ref{propbounds} for finite domains.

\begin{lemma}\label{Lem:13II19-L1}
Suppose that $0 < \omega < \frac{1}{2}\alpha$ and $R \in (0,\infty)$, and let $(\vone, \vtwo)$ be a positive solution of \eqref{Eq:Osys} in $(-R,R)$ satisfying \eqref{Eq:OBC}. Then \eqref{keybound} and \eqref{boundab} hold in  $[-R,R]$.
\end{lemma}

\begin{proof} The proof is similar to though easier than that of Proposition \ref{propbounds}, thanks to the boundary condition \eqref{Eq:OBC}. We will only give a sketch.

Let $A = \vone^2 + \vtwo^2$, $B = -\ln(\vone \vtwo)$ and define $m = \max A$ and $n = \max B$.\footnote{Note that in the notation of \eqref{Eq:26VI19-R1}, we have $m = m_*$ and $n = n_*$ thanks to \eqref{Eq:OBC}.} If $m$ is attained at the endpoints, we have $m \leq 1$. Otherwise, $m = A(x_0)$ for some $x_0 \in (-R,R)$. We then have
\begin{align*}
0
	&\geq A''(x_0)
		\geq 2\vone(x_0) \RHSA(\vone(x_0), \vtwo(x_0)) + 2\vtwo(x_0) \RHSB(\vone(x_0), \vtwo(x_0))\\
	&\geq 2A^2(x_0) - 2A(x_0)+ 4s_*  \text{ where } s_* := \min_{t \geq e^{-n}} (\alpha t^2 - \omega t).
\end{align*}
In either case, we obtain
\begin{equation}
m \leq \frac{1}{2} (1 + \sqrt{1 - 8s_*}).
	\label{Eq:17IV19-X1}
\end{equation}
Likewise, we have
\begin{equation}
n \leq - \ln \frac{\omega}{\alpha + 2 - \frac{2}{m}}.
	\label{Eq:17IV19-X2}
\end{equation}
We can now follows exactly the arguments in Steps 3 and 4 of the proof of Proposition \ref{propbounds} to reach the conclusion. We omit the details.
\end{proof}

\begin{proof}[Proof of Lemma \ref{Lem:EUMODE-FR}]
Note that \eqref{Eq:Osys} is the Euler-Lagrange equation for the functional
\[
I[\vone,\vtwo] = \int_{-R}^R \Big[\frac{1}{2} (|\vone'|^2 + |\vtwo'|^2) + \frac{1}{4} (1 - \vone^2 - \vtwo^2 )^2 + \frac{\alpha}{2} (\vone\vtwo - \frac{\omega}{\alpha})^2\Big]\,dx.
\]
The existence of a positive solution to \eqref{Eq:Osys} satisfying \eqref{Eq:OBC} follows from a simple variational argument.

The uniqueness follows from the sliding method as we have seen earlier. Suppose that $(\vone,\vtwo)$ and $(\bar\vone,\bar\vtwo)$ are positive solutions of  \eqref{Eq:Osys} in $(-R,R)$ satisfying \eqref{Eq:OBC}. Extend $(\vone,\vtwo)$ to the whole of $\RR$ by defining $(\vone,\vtwo) \equiv (a,b)$ on $(-\infty,-R)$ and  $(\vone,\vtwo) \equiv (b,a)$ on $(R,\infty)$. Let
\begin{multline*}
T = \inf\{t \in [0,2R]: \bar \vone(x) \leq \vone(x + s) \\
	\text{ and } \bar \vtwo(x) \geq \vtwo(x + s)
	 \text{ for all } x \in [-R,R], t \leq s \leq 2R\}.
\end{multline*}
$T$ is well-defined thanks to Lemma \ref{Lem:13II19-L1}. To conclude it suffices to show that $T = 0$.

Set $\tilde \vone(x) = \vone(x + T)$ and $\tilde \vtwo(x) = \vone(x + T)$. Note that $\tilde \vone \geq \bar \vone$, $\tilde \vtwo \leq \bar \vtwo$, $(\tilde \vone, \tilde \vtwo)$ is also a solution to \eqref{Eq:Osys} in the interval $(-R,R-T)$, and, in view of \eqref{keybound}, we have as before that
\begin{align*}
\tilde \vone''
	&\leq \RHSA(\tilde \vone, \bar \vtwo), \qquad \bar \vone'' = \RHSA (\bar \vone, \bar \vtwo)\\
\tilde \vtwo
	&\geq \RHSB(\bar \vone, \tilde \vtwo), \qquad \bar \vtwo'' = \RHSB(\bar \vone, \bar \vtwo).
\end{align*}
In particular, if $T$ was positive, it would follow from the strong maximum principle and the Hopf lemma that there would exist some small $\eps > 0$ such that
\[
\bar \vone(x) \leq \vone(x + s)
	\text{ and } \bar \vtwo(x) \geq \vtwo(x + s)
	 \text{ for all } x \in [-R,R], T - \eps \leq s \leq T,
\]
which would contradict the definition of $T$. We hence have $T = 0$, as desired.
\end{proof}

%----------------------------------------------------------------------------%

%----------------------------------------------------------------------------%

\appendix

\section{Appendix: proof of Lemma \ref{Lem:HypExpDecay.Pos}.}\label{App:Asym}

We now prove of the exponential decay of solutions $(\vone,\vtwo)$ to \eqref{Eq:Osys}-\eqref{Eq:26VI19-OBC} to constants. This was needed in the proof of Proposition \ref{Thm:EUMODE}. We perform a standard asymptotic analysis near a hyperbolic critical point of ODEs.

 We write $\vone = b - \hat \vone$ and $\vtwo = a + \hat \vtwo$. The system \eqref{Eq:Osys} becomes
\begin{align}
\hat \vone''
	&= - \RHSA(b - \hat \vone, a + \hat \vtwo) =: \hat \RHSA(\hat \vone, \hat \vtwo),
	\label{Eq:01II19-A1}\\
\hat \vtwo''
	&= \RHSB(b - \hat \vone, a + \hat \vtwo) =: \hat \RHSB(\hat \vone, \hat \vtwo).
	\label{Eq:01II19-A2}
\end{align}
The functions $\hat \RHSA$ and $\hat \RHSB$ are polynomials and a direct computation gives
\[
\frac{\partial(\hat \RHSA,\hat \RHSB)}{\partial (\hat \vone, \hat \vtwo)} (0,0)
	= \left[\begin{array}{cc}
	2b^2 + \alpha\,a^2 &  - \frac{\omega(2 + \alpha)}{\alpha}\\
	- \frac{\omega(2 + \alpha)}{\alpha} &2a^2 +  \alpha b^2
	\end{array} \right] =: \mathbf{A}.
\]
In particular, we have, for large $x$, that
\begin{align}
\hat \vone''
	&= (2b^2 + \alpha\,a^2) \hat \vone   - \frac{\omega(2 + \alpha)}{\alpha} \hat \vtwo + O(|\hat \vone|^2 + |\hat \vtwo|^2),
	\label{Eq:01II19-A1X}\\
\hat \vtwo''
	&= - \frac{\omega(2 + \alpha)}{\alpha} \hat \vone + (2a^2 +  \alpha b^2)\hat \vtwo + O(|\hat \vone|^2 + |\hat \vtwo|^2).
	\label{Eq:01II19-A2X}
\end{align}

The matrix $\mathbf{A}$ has two positive eigenvalues $\lambda_\pm^2$ (see \eqref{Eq:04II19-A1}). For $\omega > 0$, an $\mathbf{A}$-eigenbasis of $\RR^2$ can be chosen as $(1,\mu)$ and $(-\mu,1)$ (which correspond to the eigenvalues $\lambda_-^2$ and $\lambda_+^2$, respectively) where $\mu$ is defined in \eqref{Eq:mudef}. Note that
\[
\lim_{\omega \rightarrow 0} \mu = \left\{\begin{array}{ll}
	0 &\text{ if } \alpha > 2,\\
	\infty & \text{ if } \alpha < 2,\\
	 1	& \text{ if } \alpha = 2.
\end{array}\right.
\]
This signifies some difference in the asymptotic behavior of $(\hat \vone, \hat \vtwo)$ for $\omega = 0$ and for $\omega > 0$.

\begin{proof}[Proof of Lemma \ref{Lem:HypExpDecay.Pos}] 1. We prove \eqref{Eq:06II19-elEx} and the relation $\ell_2 = \mu \ell_1$.

As $\hat \vone(x), \hat \vtwo(x) \rightarrow 0$ as $x \rightarrow \infty$, we have from \eqref{Eq:01II19-A1}-\eqref{Eq:01II19-A2} that $\hat \vone''(x), \hat \vtwo''(x)  \rightarrow 0$ as $x \rightarrow \infty$. By interpolation, this implies $\hat \vone'(x), \hat \vtwo'(x) \rightarrow 0$ as $x \rightarrow \infty$. We now write $\hat v_3 = \hat \vone', \hat v_4 = \hat \vtwo'$, $\bhv = (\hat \vone, \hat \vtwo, \hat v_3, \hat v_4)$ and recast \eqref{Eq:01II19-A1}-\eqref{Eq:01II19-A2} as a first order system
\begin{equation}
\bhv' = \mathbf{M} \bhv + \mathbf{\hat f}(\bhv)
	\label{Eq:02II19-X2}
\end{equation}
where
\[
\mathbf{M} = \left[\begin{array}{cccc}
0 & 0 & 1 & 0\\
0 & 0 & 0 & 1\\
2b^2 + \alpha\,a^2 & - \frac{\omega(2 + \alpha)}{\alpha}  & 0 & 0\\
- \frac{\omega(2 + \alpha)}{\alpha}  & 2a^2 +  \alpha b^2 & 0 & 0
\end{array}\right],
\]
and $\mathbf{\hat f}$ is a polynomial satisfying $\mathbf{\hat f}(0) = 0$ and $D\mathbf{\hat f}(0) = 0$. Note that, as $0 < \frac{\omega}{\alpha} < \frac{1}{2}$, $\mathbf{M}$ has real and nonzero eigenvalues $\lambda_1 = -\lambda_+ < \lambda_2 = -\lambda_- < 0 < \lambda_3 = \lambda_- < \lambda_4 = \lambda_+$. Hence the origin is a hyperbolic critical point of \eqref{Eq:02II19-X2}. As $\bhv(x) \rightarrow 0$ as $x \rightarrow 0$, we thus have that, for all large $x$, $\bhv(x)$ belongs to the stable manifold of \eqref{Eq:02II19-X2} at the origin. By the Stable Manifold Theorem (see e.g. \cite[Chapter 13, Theorem 4.3]{CoddingtonLevinson}), we then have that $\bhv(x)$ converges exponentially to $0$ as $x \rightarrow \infty$ and the rate of convergence is $O(e^{-\lambda x})$ for any $0 < \lambda <  |\lambda_2| = \lambda_-$.

Set
\begin{align}
X
	&= \hat \vone + \mu \hat \vtwo,\label{Eq:06II19-X1}\\
Y
	&= -\mu \hat \vone + \hat \vtwo.\label{Eq:06II19-X2}
\end{align}
We have
\begin{align}
X''
	&= \RHSA(\hat \vone, \hat \vtwo) + \mu \RHSB(\hat \vone,\hat \vtwo)
	= \lambda_-^2 X + O(|\hat \vone|^2 + |\hat \vtwo|^2),\label{Eq:06II19-X3}\\
Y''
	&= -\mu \RHSA(\hat \vone, \hat \vtwo) +  \RHSB(\hat \vone,\hat \vtwo)
	= \lambda_+^2 Y + O(|\hat \vone|^2 + |\hat \vtwo|^2).\label{Eq:06II19-X4}
\end{align}
Applying \cite[Chapter 13, Theorem 4.5]{CoddingtonLevinson}, we can find a constant $k$ and some $\delta > 0$ such that
\begin{equation}
X = k e^{-\lambda_- x} + O(e^{-(\lambda_- + \delta) x})
\text{ and } Y = O(e^{-(\lambda_- + \delta) x}).
	\label{Eq:27VI19-A1}
\end{equation}
Assertion \eqref{Eq:06II19-elEx} and the relation $\ell_2 = \mu \ell_1$ are readily seen.

\medskip
\noindent
2. We next show that $\ell_1$ and $\ell_2$ are positive.
\medskip

Suppose  by contradiction the assertion does not hold. As $\ell_2 = \mu\ell_1$, one has $\ell_1 = \ell_2 = 0$. Returning to \eqref{Eq:27VI19-A1} we have that
\[
|X| + |Y| = O(e^{-(\lambda_- + \delta) x})
\]
which implies
\[
\limsup_{x \rightarrow \infty} \frac{\ln(|X| + |Y|)}{x} \leq -(\lambda_- + \delta).
\]
Appealing again to \cite[Chapter 13, Theorem 4.3]{CoddingtonLevinson}, we thus have
\[
\limsup_{x \rightarrow \infty} \frac{\ln(|X| + |Y|)}{x} \leq -\lambda_+.
\]
This leads to
\begin{equation}
\lim_{x \rightarrow \infty} (b - \vone(x))e^{\lambda x} = \lim_{x \rightarrow \infty} (\vtwo(x) - a)e^{\lambda x} = 0 \text{ for all } 0 < \lambda < \lambda_+.
	\label{Eq:06II19-FastRate}
\end{equation}
As as $b\sqrt{2} < \lambda_+$, this gives a contradiction to Lemma \ref{Lem:wv1As} below and so concludes the proof.
\end{proof}

\begin{lemma}\label{Lem:wv1As}
Suppose that $0 \le \omega < \frac{1}{2}\alpha$ and let $(\vone, \vtwo)$ be a positive solution of \eqref{Eq:Osys}-\eqref{Eq:26VI19-OBC}. There is some $C > 0$ such that
\begin{equation}
b - \vone(x) \geq \frac{1}{C} e^{- b\sqrt{2} x} \text{ for large } x.
	\label{Eq:03II19-T1}
\end{equation}
\end{lemma}

\begin{proof} We note from \eqref{Eq:Osys} and \eqref{keybound} that
\begin{equation}
\vone'' \stackrel{ \eqref{Eq:Osys}, \eqref{keybound}}{\geq} \vone(\vone^2 + \vtwo^2 - 1) \stackrel{\eqref{keybound}}{\geq} \vone(\vone^2 - b^2).
	\label{Eq:03II19-T2}
\end{equation}

Now, take some $R > 0$ such that $\vone(R) \geq \frac{b}{\sqrt{3}}$ in $(R,\infty)$. Select $x_0$ such that $b\tanh(\frac{b}{\sqrt{2}}(R - x_0)) = \vone(R)$. To prove \eqref{Eq:03II19-T1}, it suffices to show that
$$
\vone(x) \leq b\tanh(\frac{b}{\sqrt{2}}(x - x_0)) \text{ in }(R,\infty).
$$
To this end, we note that the function $w_{1,c} = b\tanh(\frac{b}{\sqrt{2}}(x - x_0)) + c$ satisfies for $c \geq 0$,
\begin{equation}
w_{1,c}'' = w_{1,0}(w_{1,0}^2 - b^2) \leq w_{1,c}(w_{1,c}^2 - b^2) \text{ in } (R,\infty),
	\label{Eq:03II19-T3}
\end{equation}
where we have used $w_{1,0}(x) \geq w_{1,0}(R) = \vone(R) \geq \frac{b}{\sqrt{3}}$.

Clearly there is some large $c > 0$ such that $w_{1,c} \geq \vone$ in $[R,\infty)$. Let
\[
\underline{c} = \inf\{ c \geq 0: w_{1,c} \geq \vone \text{ in }[R,\infty)\}.
\]
If $\underline{c} > 0$, then we have $w_{1,\underline{c}} \geq \vone$ in $[R,\infty)$, $w_{1,\underline{c}}(R) > \vone(R)$, $\lim_{x \rightarrow \infty} (w_{1,\underline{c}} - \vone) > 0$, and there is some $x_1 \in (R,\infty)$ such that $w_{1,\underline{c}}(x_1) = \vone(x_1)$, which gives a contradiction to the strong maximum principle, in view of \eqref{Eq:03II19-T2} and \eqref{Eq:03II19-T3}. We thus have that $\underline{c} = 0$, which implies that $w_{1,0} \geq \vone$ in $[R,\infty)$, which gives \eqref{Eq:03II19-T1}.
\end{proof}

The asymptotic behavior changes somewhat in the case $\omega = 0$, which we record here for comparison. (This is not used in the paper.)

\begin{lemma}\label{Lem:HypExpDecay.Zero}
Suppose $\alpha > 0$ and $\omega = 0$. Then, $a = 0$ and $b = 1$. Let $(\vone, \vtwo)$ be a positive solution of \eqref{Eq:Osys}-\eqref{Eq:26VI19-OBC}. Then the following statements hold.
\begin{enumerate}[(i)]
\item $\tilde\ell_2 := \lim_{x \rightarrow \infty} (\vtwo(x) - a)e^{\sqrt{\alpha} x}$ exists and is positive.
\item If $\alpha > \frac{1}{2}$, then $\tilde\ell_1^+ := \lim_{x \rightarrow \infty} (b- \vone(x))e^{\sqrt{2} x}$ exists and is positive.
\item If $\alpha < \frac{1}{2}$, then $\tilde\ell_1^- := \lim_{x \rightarrow \infty} (b- \vone(x))e^{2\sqrt{\alpha} x}$ exists and is equal to $\frac{(\alpha + 1) \tilde \ell_2^2}{2(1 - 2\alpha)}$.
\item If $\alpha = \frac{1}{2}$, then $\tilde\ell_1^* := \lim_{x \rightarrow \infty} (b- \vone(x)) \frac{e^{\sqrt{2} x}}{x}$ exists and is equal to $\frac{3 \tilde \ell_2^2}{4 \sqrt{2}}$.
\end{enumerate}
\end{lemma}

\begin{proof} The proof is similar to that of Lemma \ref{Lem:HypExpDecay.Pos} and is omitted.
\end{proof}

\bibliography{ANS}{}

\begin{thebibliography}{10}

\bibitem{am}
{\sc A.~Aftalion and P.~Mason}, {\em {R}abi-coupled two-component
  {B}ose-{E}instein condensates: Classification of the ground states, defects,
  and energy estimates}, Phys. Rev. A, 94 (2016), p.~023616.

\bibitem{aftsour17}
{\sc A.~Aftalion and C.~Sourdis}, {\em Interface layer of a two-component
  {B}ose–-{E}instein condensate}, Commun. Contemp. Math., 19 (5) (2017).

\bibitem{aftsour}
\leavevmode\vrule height 2pt depth -1.6pt width 23pt, {\em Phase transition in
  a {R}abi coupled two-component {B}ose--{E}instein condensate}, Nonlinear
  Analysis, 184 (2019), pp.~381--397.

\bibitem{alama}
{\sc S.~Alama, L.~Bronsard, A.~Contreras, and D.~Pelinovsky}, {\em Domain walls
  in the coupled {G}ross--{P}itaevskii equations}, Archive for Rational
  Mechanics and Analysis, 215 (2015), pp.~579--610.

\bibitem{BCN97}
{\sc H.~Berestycki, L.~Caffarelli, and L.~Nirenberg}, {\em Monotonicity for
  elliptic equations in unbounded {L}ipschitz domains.}, Comm. Pure Appl.
  Math., L (1997), pp.~1089--1111.

\bibitem{beres1}
{\sc H.~Berestycki, T.-C. Lin, J.~Wei, and C.~Zhao}, {\em On phase-separation
  models: asymptotics and qualitative properties}, Archive for Rational
  Mechanics and Analysis, 208 (2013), pp.~163--200.

\bibitem{BN90}
{\sc H.~Berestycki and L.~Nirenberg}, {\em Some qualitative properties of
  solutions of semilinear elliptic equations in cylindrical domains}, in
  Analysis, et cetera, Academic Press, Boston, MA, 1990, pp.~115--164.

\bibitem{BNsliding}
{\sc H.~Berestycki and L.~Nirenberg}, {\em On the method of moving planes and
  the sliding method}, Boletim da Sociedade Brasileira de
  Matem{\'a}tica-Bulletin/Brazilian Mathematical Society, 22 (1991), pp.~1--37.

\bibitem{beres2}
{\sc H.~Berestycki, S.~Terracini, K.~Wang, and J.~Wei}, {\em On entire
  solutions of an elliptic system modeling phase separations}, Advances in
  Mathematics, 243 (2013), pp.~102--126.

\bibitem{Brezis}
{\sc H.~Brezis}, {\em Semilinear equations in {$\mathbb{R}^n$} without
  condition at infinity}, Applied Mathematics and Optimization, 12 (1984),
  pp.~271--282.

\bibitem{CoddingtonLevinson}
{\sc E.~A. Coddington and N.~Levinson}, {\em Theory of ordinary differential
  equations}, McGraw-Hill Book Company, Inc., New York-Toronto-London, 1955.

\bibitem{dror}
{\sc N.~Dror, B.~A. Malomed, and J.~Zeng}, {\em Domain walls and vortices in
  linearly coupled systems}, Physical Review E, 84 (2011), p.~046602.

\bibitem{Fa99}
{\sc A.~Farina}, {\em Symmetry for solutions of semilinear elliptic equations
  in {$\mathbb{R}^N$} and related conjectures}, Ricerche Mat., 48 (suppl.)
  Papers in memory of Ennio De Giorgi. (1999), pp.~129--154.

\bibitem{FaSYM}
\leavevmode\vrule height 2pt depth -1.6pt width 23pt, {\em Some symmetry
  results for entire solutions of an elliptic system arising in phase
  separation}, Discrete Contin. Dyn. Syst. A, 34(6) (2014), pp.~2505--2511.

\bibitem{FMS}
{\sc A.~Farina, L.~Montoro, and B.~Sciunzi}, {\em Monotonicity and
  one-dimensional symmetry for solutions of {$-\Delta_p u= f (u)$} in
  half-spaces}, Calculus of Variations and Partial Differential Equations, 43
  (2012), pp.~123--145.

\bibitem{FSS}
{\sc A.~Farina, B.~Sciunzi, and N.~Soave}, {\em Monotonicity and rigidity of
  solutions to some elliptic systems with uniform limits}, 2017, to appear in
  Communications in Contemporary Mathematics, DOI: 10.1142/S0219199719500445.

\bibitem{FaSo}
{\sc A.~Farina and N.~Soave}, {\em Monotonicity and {$1$}-dimensional symmetry
  for solutions of an elliptic system arising in {B}ose-{E}instein
  condensation}, Arch. Ration. Mech. Anal., 213 (2014), pp.~287--326.

\bibitem{sp}
{\sc S.~Lellouch, T.-L. Dao, T.~Koffel, and L.~Sanchez-Palencia}, {\em
  Two-component {B}ose gases with one-body and two-body couplings}, Phys. Rev.
  A, 88 (2013), p.~063646.

\bibitem{matexp}
{\sc M.~Matthews, B.~Anderson, P.~Haljan, D.~Hall, M.~Holland, J.~Williams,
  C.~Wieman, and E.~Cornell}, {\em Watching a superfluid untwist itself:
  Recurrence of {R}abi oscillations in a {B}ose-{E}instein condensate}, Phys.
  Rev. Lett., 83 (1999), p.~3358.

\bibitem{qustring}
{\sc C.~Qu, M.~Tylutki, S.~Stringari, and L.~Pitaevskii}, {\em Magnetic
  solitons in {R}abi-coupled {B}ose-{E}instein condensates}, Physical Review A,
  95 (2017), p.~033614.

\bibitem{sourdisweak}
{\sc C.~Sourdis}, {\em On the weak separation limit of a two-component
  bose-einstein condensate}, arXiv preprint arXiv:1611.04470,  (2016).

\bibitem{usui}
{\sc A.~Usui and H.~Takeuchi}, {\em Rabi-coupled countersuperflow in binary
  {B}ose-{E}instein condensates}, Phys. Rev. A, 91 (2015), p.~063635.

\end{thebibliography}
\bibliographystyle{siam}
\end{document}